\documentclass[11pt]{article}
\usepackage{fullpage}
\usepackage{amsmath,amssymb,amsthm}
\usepackage[english]{babel}
\usepackage[left=3cm,right=3cm,top=3cm,bottom=4cm]{geometry}

\usepackage{tikz,tikz-cd}
\usetikzlibrary{positioning}
\usepackage{braket,mathtools,stmaryrd}

\SetSymbolFont{stmry}{bold}{U}{stmry}{m}{n}
\usepackage{lmodern,bm,bbm,mathrsfs,manfnt}
\usepackage{enumitem}

\usepackage[colorinlistoftodos]{todonotes}

\usepackage[backref=page, colorlinks=true, linkcolor=blue, citecolor=red, menucolor=green]{hyperref}
\usepackage{setspace}

\theoremstyle{plain}
\newtheorem*{theorem}{Theorem}
\newtheorem*{proposition}{Proposition}
\newtheorem*{lemma}{Lemma}
\newtheorem*{corollary}{Corollary}

\theoremstyle{definition}
\newtheorem*{definition}{Definition}
\newtheorem*{example}{Example}
\newtheorem*{remark}{Remark}

\numberwithin{equation}{section}

\renewcommand{\bar}[1]{\overline{#1}}
\newcommand{\cat}[1]{\mathsf{#1}}
\renewcommand{\hat}[1]{\widehat{#1}}
\newcommand{\lie}[1]{\mathfrak{#1}}
\renewcommand{\tilde}[1]{\widetilde{#1}}
\newcommand{\bC}{\mathbb{C}}
\newcommand{\bh}{\mathbbm{h}}
\newcommand{\bk}{\mathbbm{k}}
\newcommand{\bL}{\mathbb{L}}
\newcommand{\bQ}{\mathbb{Q}}

\newcommand{\bT}{\mathbb{T}}
\newcommand{\bZ}{\mathbb{Z}}
\newcommand{\cA}{\mathcal{A}}
\newcommand{\cB}{\mathcal{B}}
\newcommand{\cE}{\mathcal{E}}
\newcommand{\cF}{\mathcal{F}}
\newcommand{\cG}{\mathcal{G}}
\newcommand{\cH}{\mathcal{H}}
\newcommand{\cL}{\mathcal{L}}
\newcommand{\cN}{\mathcal{N}}
\newcommand{\cO}{\mathcal{O}}
\newcommand{\cT}{\mathcal{T}}
\newcommand{\cU}{\mathcal{U}}
\newcommand{\cV}{\mathcal{V}}
\newcommand{\fC}{\mathfrak{C}}
\newcommand{\fM}{\mathfrak{M}}
\newcommand{\fN}{\mathfrak{N}}
\newcommand{\fU}{\mathfrak{U}}
\newcommand{\fV}{\mathfrak{V}}
\newcommand{\fW}{\mathfrak{W}}
\newcommand{\fX}{\mathfrak{X}}
\newcommand{\fY}{\mathfrak{Y}}
\newcommand{\fZ}{\mathfrak{Z}}
\newcommand{\sA}{\mathsf{A}}
\newcommand{\sG}{\mathsf{G}}
\newcommand{\sH}{\mathsf{H}}
\newcommand{\sR}{\mathsf{R}}
\newcommand{\sT}{\mathsf{T}}
\newcommand{\sU}{\mathsf{U}}
\newcommand{\Y}{\rotatebox[origin=c]{180}{$Y$}}
\newcommand{\alg}{\mathrm{alg}}
\newcommand{\cl}{\mathrm{cl}}
\newcommand{\doub}{\mathrm{doub}}
\newcommand{\loc}{\mathrm{loc}}
\newcommand{\nil}{\mathrm{nil}}
\newcommand{\op}{\mathrm{op}}
\newcommand{\pt}{\mathrm{pt}}
\newcommand{\sg}{\mathrm{sg}}
\newcommand{\trip}{\mathrm{trip}}
\newcommand{\vac}{\mathbf{1}}

\DeclareMathOperator{\End}{End}

\DeclareMathOperator{\coker}{coker}
\DeclareMathOperator{\crit}{crit}
\DeclareMathOperator{\Ext}{Ext}
\DeclareMathOperator{\Frac}{Frac}
\DeclareMathOperator{\GL}{GL}
\DeclareMathOperator{\Gr}{Gr}
\DeclareMathOperator{\gr}{gr}
\DeclareMathOperator{\Hom}{Hom}
\DeclareMathOperator{\cHom}{\mathcal{H}{\it om}}
\DeclareMathOperator{\id}{id}
\DeclareMathOperator{\im}{im}
\DeclareMathOperator{\rank}{rank}
\DeclareMathOperator{\Spec}{Spec}
\DeclareMathOperator{\Sym}{Sym}
\DeclareMathOperator{\tr}{tr}
\DeclarePairedDelimiterX{\lseries}[1]{(}{)}{\mkern-2mu\delimsize(#1\delimsize)\mkern-2mu}
\DeclarePairedDelimiterX{\pseries}[1]{[}{]}{\mkern-2mu\delimsize[#1\delimsize]\mkern-2mu}

\title{Multiplicative vertex algebras and quantum loop algebras}
\author{Henry Liu}
\date{\today}

\begin{document}

\maketitle

\begin{abstract}
  We define a multiplicative version of vertex coalgebras and show
  that various equivariant K-theoretic Hall algebras (KHAs) admit
  compatible multiplicative vertex coalgebra structures. In
  particular, this is true of Varagnolo--Vasserot's preprojective KHA,
  which is (conjecturally) isomorphic to positive halves of certain
  quantum loop algebras.
\end{abstract}

\begin{spacing}{1.15}
  \tableofcontents
\end{spacing}

\renewcommand{\thefootnote}{\fnsymbol{footnote}} 
\renewcommand{\thefootnote}{\arabic{footnote}}

\section{Introduction}

In \cite{Joyce2021}, Joyce geometrically constructs a vertex algebra
structure on the homology groups of certain moduli stacks $\fM$. In
\cite{Liu2022}, we gave an equivariant and multiplicative
generalization: the so-called operational K-homology groups of $\fM$
are equivariant multiplicative vertex algebras. In particular, this
holds when $\fM = \fM_Q$ is a moduli stack of representations of a
quiver $Q$. The cohomology/K-theory/etc. groups of stable loci of
$\fM_Q$ often carry actions of Yangians/quantum loop algebras/etc.
\cite{Davison2023,Maulik2019}. It is then natural to ask: what is the
interaction between the multiplicative vertex algebras and these
quantum loop algebras?

Contrary to this question and the title, in this paper there are no
multiplicative vertex algebras. Rather, we define
(\S\ref{sec:braided-coVA}) and study the categorically-dual notion of
(braided) multiplicative vertex {\it coalgebras}. Their axioms are
different from the naive multiplicative analogue of ordinary vertex
coalgebra axioms \cite{Hubbard2009}, and also different from the
categorical dual of the vertex $F$-algebras of \cite{Li2011} when $F$
is the multiplicative group law. For instance, there appears to be no
canonical notion of an ``unbraided'' multiplicative vertex coalgebra.

The geometric input is as follows. The moduli stack $\fM_Q$ has a
natural action by a torus $\sT$ scaling the linear maps in the
representation given by edges of $Q$. We consider the equivariant
(algebraic, zeroth) K-group $K_\sT(\fM_Q)$. Following the well-known
Kontsevich--Soibelman construction in cohomology
\cite{Kontsevich2011}, $K_\sT(\fM_Q)$ can be made into a {\it
  K-theoretic Hall algebra (KHA)} with product denoted by $\star$.

\begin{theorem}[Easy case of main theorems]
  \begin{enumerate}
  \item (Theorem~\ref{thm:coVA-construction-general}) $K_\sT(\fM_Q)$
    admits a multiplicative vertex coalgebra structure $(\vac, D(z),
    \Y(z), C(z))$.

  \item (Theorem~\ref{thm:compatibility-general}) The KHA product
    $\star$ on $K_\sT(\fM_Q)$ is compatible with this multiplicative
    vertex coalgebra structure, forming a multiplicative vertex
    bialgebra.
  \end{enumerate}
\end{theorem}

Both parts of this theorem are direct K-theoretic analogues of
cohomological results of Latyntsev \cite{Latyntsev2021}. The first
part uses a construction dual to the K-homology construction of
\cite{Liu2022}. Roughly, the vertex coproduct $\Y(z)$ is given by
pullback along the direct sum map $\Phi\colon \fM_Q \times \fM_Q \to
\fM_Q$, followed by a twist involving a perfect complex $\cE \in
K_\sT^\circ(\fM_Q \times \fM_Q)$ with specific bilinearity properties.
The same twist is used to construct the half-braiding operator $C(z)$.
We expect the vertex coalgebra structure to enrich the study of the
representation theory of KHAs. Furthermore, it should be much easier
to study the vertex coalgebra $K_\sT(\fM_Q)$ than the vertex algebras
present in \cite{Liu2022}.

The relation of all this to quantum loop algebras appears from the
same constructions and results, but for the cotangent or {\it
  preprojective} stack $T^*\fM_Q$. By work of Varagnolo and Vasserot
\cite{Varagnolo2022}, $K_\sT(T^*\fM_Q)$ is also a KHA, called the {\it
  preprojective KHA}. They conjecture, and prove when $Q$ is finite or
affine type excluding $A_1^{(1)}$, an isomorphism
\[ K_\sT(T^*\fM_Q) \cong \cU_\hbar^+(L\lie{g}_{\text{MO}}) \]
with the positive part of the quantum loop algebra constructed by
Maulik, Okounkov and Smirnov \cite{Maulik2019,Okounkov2022} using
(K-theoretic) stable envelopes on the Nakajima quiver varieties
associated to $Q$. This is doubly interesting because the KHA product
$\star$, of arbitrary elements, has the very explicit form of a {\it
  shuffle product} \cite{Negut2021}. For compatibility, the kernel of
this shuffle product must be exactly the bilinear element $\cE$
defining the vertex coalgebra.

\begin{theorem}[Main theorems]
  \begin{enumerate}
  \item (Theorem~\ref{thm:coVA-construction-preprojective})
    $K_\sT(T^*\fM_Q)_{\loc}$ admits a multiplicative vertex coalgebra
    structure $(\vac, D(z), \Y(z), C(z))$.

  \item (Theorem~\ref{thm:compatibility-preprojective}) There is a
    twisted KHA product $\star_\omega$ on $K_\sT(T^*\fM_Q)_{\loc}$
    compatible with this multiplicative vertex coalgebra structure,
    forming a multiplicative vertex bialgebra.
  \end{enumerate}
\end{theorem}

Unlike $\fM_Q$, the stack $T^*\fM_Q$ is badly singular, and so the
main technical difficulty here is that pullback along $\Phi$ no longer
exists on $K_\sT(T^*\fM_Q)$ and cannot be used to construct a vertex
coproduct. However, by dimensional reduction \cite{Isik2013}, there is
an isomorphism
\begin{equation} \label{eq:preprojective-to-critical-K-groups-intro}
  K_\sT(T^*\fM_Q) \cong K_\sT^{\crit}(\fM_{Q^{\trip}}, \tr W^{\trip})
\end{equation}
with the equivariant {\it critical K-groups} of the {\it tripled}
quiver $Q^{\trip}$ associated to $Q$ and an appropriate potential
$W^{\trip}$ on $Q^{\trip}$. Roughly, if the ambient space $M$ is
affine, $K_\sT^{\crit}(M, \phi)$ is a better-behaved refinement of the
ordinary K-theory of the critical locus $\{d\phi = 0\} \subset M$. By
virtue of its presentation as K-groups of matrix factorizations when
$M$ is smooth \cite{Orlov2004,Polishchuk2011}, critical K-theory
admits pullbacks along arbitrary maps, including $\Phi$, which we use
to construct the desired vertex coalgebra structure.

In fact, the main theorems hold very generally: for arbitrary quivers
with potential $(Q, W)$, P{\u a}durariu constructs a KHA structure on
$K_\sT^{\crit}(\fM_Q, \tr W)$ \cite{Puadurariu2023}, and we can make
these {\it critical KHAs} into vertex bialgebras as well
(Remarks~\ref{rem:coVA-general}, \ref{rem:vertex-bialgebra-general}).
But a mild K\"unneth assumption is required, and unlike in ordinary
cohomology, K\"unneth theorems are rare in K-theory, especially
equivariantly where some form of equivariant formality is usually
necessary. We discuss this in Appendix~\ref{sec:kunneth-properties}.
For this reason, and also for simplicity of exposition, the main
theorems are stated only for the special case of
\eqref{eq:preprojective-to-critical-K-groups-intro}.

It is very plausible that all of our results continue to hold in the
world of ordinary vertex coalgebras, critical cohomology,
cohomological Hall algebras, and Yangians. Indeed, many of our
constructions, especially for $T^*\fM_Q$, stem from earlier
cohomological work of Davison, see e.g. \cite{Davison2023}.
Furthermore, our results should also generalize immediately to the
K-theory of moduli stacks of coherent sheaves on curves. (Surfaces may
be harder because an analogue of
\eqref{eq:preprojective-to-critical-K-groups-intro} is needed.)

\subsection{Outline of the paper}

We begin in \S\ref{sec:equivariant-K-theory} with a leisurely review
of equivariant K-theory, both the ordinary and the critical kind. In
\S\ref{sec:K-theory}, we fix some notation and provide some tools for
equivariant K-theory in general. In particular we review
(\S\ref{sec:virtual-localization}) virtual localization in the
language of dg-schemes. In \S\ref{sec:critical-K-theory}, we define
critical K-theory as the K-group of a specific singularity category,
and explain its presentation using matrix factorizations
(Theorem~\ref{thm:Dcrit-as-matrix-factorizations}) as well as its
dimensional reduction theorem
(Theorem~\ref{thm:dimensional-reduction}) which at the level of
derived categories involves dg-schemes. As a fairly representative
example, we compute $K_\sT^{\crit}(\bC^2, xy)$.

Section~\ref{sec:braided-coVA} is about multiplicative vertex
coalgebras and our geometric construction of them. In
\S\ref{sec:braided-coVA-theory}, we give and motivate the general
definition, and explain why it is categorically dual to multiplicative
vertex {\it algebras}. In \S\ref{sec:quiver-coVA}, we set up moduli
stacks $\fM$ of quiver representations for a quiver $Q$, its doubling
$Q^{\doub}$ and its tripling $Q^{\trip}$, and make $K_\sT(\fM)$ into
multiplicative vertex coalgebras
(Theorem~\ref{thm:coVA-construction-general}). In
\S\ref{sec:preprojective-coVA}, we do the same for the preprojective
stack $T^*\fM_Q$ via the critical K-theory of $\fM_{Q^{\trip}}$
(Theorem~\ref{thm:coVA-construction-preprojective}). Some localization
is necessary here to preserve the K\"unneth property.

Section~\ref{sec:compatibility-KHA} upgrades these multiplicative
vertex coalgebras into multiplicative vertex bialgebras. In
\S\ref{sec:HAs}, we begin by defining the Hall product on all the
K-groups above, taking care to note a slight discrepancy
(Proposition~\ref{prop:twisted-preprojective-KHA}) between the Hall
products of the two sides of
\eqref{eq:preprojective-to-critical-K-groups-intro}. In
\S\ref{sec:compatibility}, we prove the main compatibility theorems
between the vertex coalgebra and Hall algebra structures on the
K-groups of $\fM$ and $T^*\fM$. This gives an geometric interpretation
of the formal variable $z$ appearing in the vertex coalgebra as the
weight of a certain $\bC^\times$-action. Finally, in
\S\ref{sec:comparison-ambient-vertex-bialgebra}, we show that the
natural morphism $K_\sT^{\crit}(\fM_{Q^{\trip}}, \tr W^{\trip}) \to
K_\sT(\fM_{Q^{\trip}})$, known already to be a Hall algebra morphism,
also preserves the vertex coalgebra structure. We also record
(\S\ref{sec:quiver-KHA-shuffle-formula}) some explicit formulas for
the vertex bialgebra on the right hand side.

Appendix~\ref{sec:kunneth-properties} gives a general strategy to
prove K\"unneth theorems in equivariant K-theory, using excision along
a (equivariant) stratification whose strata individually satisfy
K\"unneth theorems. While the strategy is insufficient when applied to
$K_\sT(T^*\fM)$, it does work for the related stack $\fN^{\nil}$ of
quiver representations with nilpotent endomorphism.

\subsection{Acknowledgements}

This project benefitted greatly from interactions with D. Joyce, A.
Okounkov, and T. P{\u a}durariu, and was supported by the Simons
Collaboration on Special Holonomy in Geometry, Analysis and Physics.
The revised version was written with support from the World Premier
International Research Center Initiative (WPI), MEXT, Japan.

\section{Equivariant K-theory}
\label{sec:equivariant-K-theory}

\subsection{Notation and review}
\label{sec:K-theory}

\subsubsection{}

Throughout this paper, all (dg-)schemes are separated and finite type
over $\bC$.

\subsubsection{}

\begin{definition}
  Let $X$ be a quasi-projective scheme with the action of a reductive
  group $\sG$. Let
  \begin{equation} \label{eq:perf-inside-dbcoh}
    \cat{Perf}_\sG(X) \subset D^b\cat{Coh}_\sG(X)
  \end{equation}
  be the full subcategory of $\sG$-equivariant perfect complexes,
  inside the derived category of $\sG$-equivariant coherent sheaves on
  $X$. Denote their Grothendieck K-groups by
  \begin{align*}
    K_\sG(X) &\coloneqq K_0(D^b\cat{Coh}_\sG(X)) \\
    K_\sG^\circ(X) &\coloneqq K_0(\cat{Perf}_\sG(X)).
  \end{align*}
  Equivalently, $K_\sG^\circ(X) \cong K_0(\cat{Vect}_\sG(X))$ is built
  from $\sG$-equivariant vector bundles \cite[\S 2]{Totaro2004}.

  Both $K_\sG(X)$ and $K_\sG^\circ(X)$ are modules for
  $\bk_\sG\coloneqq K_\sG(\pt)$, which by definition is the
  representation ring of $\sG$. If $\sT \subset \sG$ is a maximal
  torus, then
  \[ \bk_\sG = \bZ[t^\mu]^W \subset \bZ[t^\mu] = \bk_\sT \]
  is (the Weyl-invariant part of) the group algebra of the character
  lattice of $\sT$.
\end{definition}

\subsubsection{}

Unless stated otherwise, all pushforwards and pullbacks are derived,
and, when working with $\sG$-equivariant K-groups, all objects and
morphisms are assumed to be $\sG$-equivariant. This is to preserve
$\sG$-equivariant exact sequences, a necessary condition to induce
morphisms of $\sG$-equivariant K-groups.

\subsubsection{}
\label{sec:k-theory-pushforward-pullback}

The $\bk_\sG$-modules $K_\sG(X)$ and $K_\sG^\circ(X)$ carry different
functoriality and structure, and the inclusion
\eqref{eq:perf-inside-dbcoh} induces a morphism $\Upsilon\colon
K_\sG^\circ(X) \to K_\sG(X)$ of $\bk_\sG$-modules which is in general
neither injective nor surjective. Let $f\colon X \to Y$ be a
$\sG$-equivariant morphism.
\begin{itemize}
\item There is a (functorial) pullback $f^*\colon K^\circ_\sG(Y) \to
  K^\circ_\sG(X)$. Tensor product $\otimes\colon K^\circ_\sG(X)
  \otimes K^\circ_\sG(X) \to K^\circ_\sG(X)$ makes $K^\circ_\sG(X)$
  into a ring.

\item If $f$ is proper, there is a (functorial) pushforward $f\colon
  K_\sG(X) \to K_\sG(Y)$. If $f$ has finite Tor amplitude, e.g. $f$ is
  flat, there is a (functorial) pullback $f^*\colon K_\sG(Y) \to
  K_\sG^\circ(X)$ which we typically compose with $\Upsilon$ to get
  $f^*\colon K_\sG(Y) \to K_\sG(X)$. Tensor product $\otimes\colon
  K^\circ_\sG(X) \otimes K_\sG(X) \to K_\sG(X)$ makes $K_\sG(X)$ into
  a $K_\sG^\circ(X)$-module.

\item While the external tensor product $\boxtimes\colon K_\sG(X)
  \otimes K_\sG(Y) \to K_\sG(X \times Y)$ always exists, $\otimes
  \coloneqq \Delta^* \boxtimes$ only exists if the diagonal embedding
  $\Delta\colon X \to X \times X$ has finite Tor amplitude.
\end{itemize}
If $X$ is smooth, $\Upsilon$ is an isomorphism, see e.g.
\cite[Proposition 5.1.28]{Chriss1997}, otherwise the discrepancy is
measured by the {\it singularity category}
\[ D_\sG^{\sg}(X) \coloneqq D^b\cat{Coh}_\sG(X)/\cat{Perf}_\sG(X). \]

\subsubsection{}
\label{sec:equivariant-K-theory-properties}

We primarily use the following tools to control equivariant K-groups.

\begin{theorem}
  \begin{enumerate}
  \item (Long exact sequence {\cite[\S 5.2.14]{Chriss1997}}) If
    $i\colon Z \hookrightarrow X$ is a $\sG$-equivariant closed
    embedding, and $j\colon U \hookrightarrow X$ is its complement,
    then there is a long exact sequence
    \begin{equation} \label{eq:long-exact-sequence}
      \cdots \to K_\sG(Z) \xrightarrow{i_*} K_\sG(X) \xrightarrow{j^*} K_\sG(U) \to 0
    \end{equation}
    where $\cdots$ hides complicated beasts known as higher K-groups.

  \item (Thom isomorphism theorem {\cite[Theorem 5.4.17]{Chriss1997}})
    If $\pi\colon E \to X$ is a $\sG$-equivariant vector bundle, then
    $\pi^*\colon K_\sG(X) \to K_\sG(E)$ is an isomorphism.

  \item (Equivariant concentration {\cite[Th\'eor\`eme
      2.2]{Thomason1992}}) Let $g \in \sG$ be a central element. Then
    the inclusion $i\colon X^g \hookrightarrow X$ of the $g$-fixed
    locus induces an isomorphism
    \[ i_*\colon K_\sG(X^g)_\loc \xrightarrow{\sim} K_\sG(X)_\loc, \]
    where the subscript $\loc$ indicates base change from $\bk_\sG$ to
    $\Frac(\bk_\sG)$.
  \end{enumerate}
\end{theorem}

To emphasize, equivariant concentration holds without any further
assumptions on the closed immersion $i$. Additional assumptions, e.g.
that $i$ is regular, are only required when one wants a nice formula
for the inverse $(i_*)^{-1}$, using the self-intersection formula
\eqref{eq:self-intersection-formula} below for instance.

\subsubsection{}

\begin{remark}
  Every linear algebraic group $\sG$ decomposes as $\sG = \sR \ltimes
  \sU$ where $\sR$ is reductive and $\sU$ is its unipotent radical.
  The Thom isomorphism theorem, along with the Morita equivalence
  $K_\sG(\sG \times_\sH X) \cong K_\sH(X)$ for subgroups $\sH \subset
  \sG$, can be used to show that
  \begin{equation} \label{eq:equivariant-reduction}
    K_{\sR \ltimes \sU}(X) \cong K_{\sR}(X)
  \end{equation}
  depends only on the reductive part of $\sG$ \cite[\S
    5.2.18]{Chriss1997}.
\end{remark}

\subsubsection{}
\label{sec:dg-stuff}

For convenience later, e.g. for virtual localization
(\S\ref{sec:virtual-localization}), dimensional reduction
(\S\ref{sec:dimensional-reduction}) and base change
(\S\ref{sec:compatibility-base-change-step}) formulas, we will
occasionally work with {\it dg-schemes} $X \coloneqq (X^0,
\cO_X^\bullet)$. This means that $\cO_X^\bullet$ is a quasi-coherent
sheaf of commutative differential graded algebras (cdga) on a scheme
$X^0$, with $\cO_X^i = 0$ for $i > 0$ and $\cO_X^0 = \cO_{X^0}$. The
{\it classical truncation} of a dg-scheme $X$ is
\[ X^{\cl} \coloneqq \Spec \cH^0(\cO_X^\bullet) \subset X^0; \]
conversely, every classical scheme $X$ is a dg-scheme $(X, \cO_X)$
where $\cO_X$ sits in degree zero. A $\sG$-action on $X$ means a
$\sG$-action on $X^0$ such that $\cO_X$ has $\sG$-equivariant product
and differential. One can view $X$ as approximately equivalent to
$X^{\cl}$ equipped with a ($\sG$-equivariant) obstruction theory.

An $\cO_X$-module $\cE$ is a $\cO_{X^0}$-module with an action of the
cdga $\cO_X^\bullet$, and is {\it coherent} if its total cohomology
sheaf $\cH(\cE) \coloneqq \bigoplus_i \cH^i(\cE)[-i]$ is coherent over
$\cH(\cO_X^\bullet)$. Then $D^b\cat{Coh}(X)$ is defined to be the
derived category of the category of coherent $\cO_X$-modules, i.e. the
triangulated category obtained by inverting all quasi-isomorphisms in
the homotopy category of coherent $\cO_X$-modules \cite[\S
  1]{Isik2013}. It has a standard t-structure whose heart
$D^{\heartsuit} \subset D^b\cat{Coh}(X)$ consists of coherent
$\cO_X$-modules with cohomology only in degree $0$, so
\[ \cH^0\colon D^{\heartsuit} \xrightarrow{\sim} \cat{Coh}(X^{\cl}) \]
is an equivalence of categories. This is also true equivariantly,
hence $K_\sG(X) = K_\sG(X^{\cl})$. However, in general $\cH^0$ does
not preserve perfect complexes, so $K_\sG^\circ(X) \neq
K_\sG^\circ(X^{\cl})$.

All of the preceding content in this subsection continues to hold for
dg-schemes, without change, with basically the same proofs
\cite{Khan2022, Aranha2024a}.

\subsubsection{}

\begin{example} \label{ex:derived-zero-locus}
  Let $s \in \Gamma(X, \cE)$ be a section of a locally free sheaf on a
  scheme $X$. The {\it derived zero locus}
  $s^{-1}(0)^{\text{derived}}$ is (or has a preferred model as) the
  derived Spec
  \[ s^{-1}(0)^{\text{derived}} = R\Spec(\wedge^\bullet \cE^\vee) \]
  where $\wedge^\bullet \cE^\vee$ is the Koszul complex associated to
  $s$, and the ordinary zero locus $s^{-1}(0)$ is its classical
  truncation. If $s$ is a regular section, then the Koszul complex is
  exact except at degree $0$ and
  \[ s^{-1}(0)^{\text{derived}} = s^{-1}(0). \]
  Otherwise the two are different, and have different derived
  categories (but the same K-groups). Note that if $X$ is a smooth
  variety, $s$ is regular if and only if $s^{-1}(0)$ is of expected
  dimension.
\end{example}

\subsubsection{}

A morphism $f\colon X \to Y$ of dg-schemes has an associated
$\cO_X$-module $\bL_{X/Y}$ of K\"ahler differentials, which we view as
a complex of $\cO_{X^0}$-modules and call the {\it cotangent complex}.
When it is perfect, its dual is denoted by $\bT_{X/Y}$ and called the
{\it tangent complex}.

We say $f$ is {\it quasi-smooth} if $\bL_{X/Y}$ is perfect and of
Tor-amplitude $[-1, \infty)$. For instance, if $X$ is quasi-smooth
  (over $Y = \pt$), the map $i^*\bL_X \to \bL_{X^{\cl}}$ associated to
  the canonical inclusion $i\colon X^{\cl} \to X$ is a perfect
  obstruction theory for $X^{\cl}$. Note that, in K-theory, pullback
  along the induced morphism $f^{\cl}\colon X^{\cl} \to Y^{\cl}$ is
  generally different from pullback along $f\colon X \to Y$, which
  classically is known as a {\it virtual pullback} \cite{Qu2018}.

\subsubsection{}
  
For a vector bundle $\cE \in \cat{Vect}_\sG(X)$, let $\wedge^i \cE$ be
its $i$-th exterior power and, for a formal variable $z$, define
\[ \wedge^\bullet_{-z}(\cE) \coloneqq \sum_i (-z)^i \wedge^i\cE \in K_\sG^\circ(X)[z]. \]
When $z=1$, this is the K-theoretic analogue of the Euler class of
$\cE$: for a quasi-smooth closed immersion $i\colon Z \hookrightarrow
X$ of dg-schemes, there is the K-theoretic {\it self-intersection
  formula}
\begin{equation} \label{eq:self-intersection-formula}
  i^* i_*(-) = (-) \otimes \wedge^\bullet_{-1}(\cN_i^\vee)
\end{equation}
where $\cN_i^\vee \coloneqq \bL_i[-1]$ is the {\it virtual conormal
  bundle} of $i$ \cite[\S 2.5]{Qu2018}, and its proof shows that the
equality holds for both $i^*i_*\colon K_\sG(Z) \to K_\sG(Z)$ and
$i^*i_*\colon K_\sG^\circ(Z) \to K_\sG^\circ(Z)$. Recall that if $Z$
and $X$ are classical schemes, then $i$ is quasi-smooth if and only if
it is regular.

If $K_\sG(Z)$ is not torsion for $\wedge^\bullet_{-1}(\cN_i^\vee)$,
then $i_*$ must be injective and the long exact sequence
\eqref{eq:long-exact-sequence} becomes short exact.

\subsubsection{}
\label{sec:virtual-localization}

Let $X$ be a quasi-smooth dg-scheme acted on by $\sG$, and $i\colon
X^g \hookrightarrow X$ be the $g$-fixed locus for a central element $g
\in \sG$ \cite[\S 5.2]{Ciocan-Fontanine2009}. The self-intersection
formula does not immediately apply to $i$, because $i$ may not be
quasi-smooth or even of finite Tor amplitude. Assuming that
$\cN_i^\vee$ has a global resolution $\cE_1 \to \cE_0$ by
$\sG$-equivariant vector bundles, the typical procedure (e.g. like in
\cite[\S 3.2]{Qu2018}) is to adjust the derived structure on $X^g$ by
$\cE_1$ to make $i$ quasi-smooth, and then to apply the usual
self-intersection formula \eqref{eq:self-intersection-formula}.
Assuming furthermore that an inverse of $\wedge^\bullet_{-1}(\cE_1)$
exists in $K_\sG(X^g)_{\loc}$, this adjustment may then be reversed by
multiplying by $\wedge^\bullet_{-1}(\cE_1)^{-1}$.

For us, it will be more convenient to use the formalism of
\cite{Aranha2024}. To summarize, in the above setting with the above
assumptions, they repackage the aforementioned procedure into a
homomorphism $i^!\colon K_\sG(X)_{\loc} \to K_\sG(X^g)_{\loc}$ called
{\it Gysin pullback}, and then prove the self-intersection formula
\begin{equation} \label{eq:virtual-self-intersection-formula}
  i^! i_*(-) = (-) \otimes \wedge^\bullet_{-1}(\cN_i^\vee)
\end{equation}
on $K_\sG(X^g)_{\loc}$, where $\wedge^\bullet_{-1}(\cN_i^\vee)
\coloneqq \wedge^\bullet_{-1}(\cE_0) \otimes
\wedge^\bullet_{-1}(\cE_1)^{-1}$ is well-defined by assumption. If $i$
is quasi-smooth then $\cE_1 = 0$ and $i^! = i^*$, recovering
\eqref{eq:self-intersection-formula}, but in general $i^!$ is only
well-defined after passing to localized K-groups. In some sense, this
is because $i^!$ differs from $i^*$ by exactly the factor of
$\wedge_{-1}^\bullet(\cE_1)^{-1}$, and
\eqref{eq:virtual-self-intersection-formula} is
\eqref{eq:self-intersection-formula} with both sides multiplied by
$\wedge_{-1}^\bullet(\cE_1)^{-1}$.

Equivariant concentration says $i_*$ is invertible, so if in addition
$\wedge^\bullet_{-1}(\cE_0)$ is also invertible in
$K_\sG(X^g)_{\loc}$, then \eqref{eq:virtual-self-intersection-formula}
immediately implies the {\it virtual localization} formula
\begin{equation} \label{eq:virtual-localization}
  (i_*)^{-1} = \wedge^\bullet_{-1}(\cN_i^\vee)^{-1} \otimes i^!.
\end{equation}

\subsubsection{}
\label{sec:stabilizers-as-equivariance}

Finally, all (dg-)stacks appearing in this paper are naturally global
quotients $[X/G]$ of a quasi-projective (dg-)scheme $X$ by a reductive
group $G$, and we only consider groups $\sG$ acting on $X$ which {\it
  commute} with the $G$-action. In this setting,
\[ D^b\cat{Coh}_{\sG}([X/G]) = D^b\cat{Coh}_{\sG \times G}(X) \]
and similarly for $\cat{Perf}$. One can take the right hand side to be
the definition of the left hand side, if desired. We will often
implicitly switch between the two sides.

\subsection{Critical K-theory}
\label{sec:critical-K-theory}

\subsubsection{}

\begin{definition} \label{def:critical-K-theory}
  Let $M$ be a quasi-projective scheme acted on by a reductive group
  $\sG$, and
  \[ \phi \in \Gamma(M, \cO_M) \]
  be a $\sG$-equivariant regular function of $\sG$-weight denoted by
  $\kappa$. We call $\phi$ the {\it potential}. Assume that $0$ is the
  only critical value of $\phi$. Set
  \[ D_\sG^{\crit}(M, \phi) \coloneqq D_\sG^{\sg}(\phi^{-1}(0)). \]
  The {\it critical K-theory} of $(M, \phi)$ is
  \[ K_\sG^{\crit}(M, \phi) \coloneqq K_0(D_\sG^{\crit}(M, \phi)). \]
  This can be extended to dg-schemes $M$ and potentials $\phi \in
  \Gamma(M^{\cl}, \cO_{M^{\cl}})$, taking $\phi^{-1}(0)$ to be the
  derived zero locus. This can also be extended to quotient
  (dg-)stacks $\fM = [M/G]$ for potentials $\phi$ on $M$ which are
  $G$-invariant, following the discussion of
  \S\ref{sec:stabilizers-as-equivariance}. Note that if $M$ is a
  dg-scheme, $K_\sG^{\crit}(M, \phi) \neq K_\sG^{\crit}(M^{\cl},
  \phi)$ in general, cf. \S\ref{sec:dg-stuff}.
\end{definition}

\subsubsection{}

For most of this paper, $M$ will be affine. Then elements of
$D_\sG^{\crit}(M, \phi)$ are supported only on the singular locus
$\crit(\phi) \coloneqq \{d\phi = 0\} \subset \phi^{-1}(0)$, where
$D^b\cat{Coh}$ and $\cat{Perf}$ differ. Hence $D_\sG^{\crit}(-)$ can
be viewed as a refinement of $D^b\cat{Coh}_\sG(\crit(-))$, see e.g.
\cite{Teleman2020} and \S\ref{sec:Dcrit-functoriality}, and it
categorifies many aspects of critical cohomology.

For us, it will be more useful to consider the following presentation
of $D_\sG^{\crit}(M, \phi)$ as a category of matrix factorizations.

\subsubsection{}

\begin{definition}[{\cite[\S 3.1]{Orlov2004}}]
  Let $M$ be a smooth quasi-projective scheme acted on by a reductive
  group $\sG$. A $\sG$-equivariant {\it matrix factorization} of
  $\phi$ is a pair
  \begin{equation} \label{eq:matrix-factorization}
    \cE_1 \xrightarrow{d_1} \cE_0 \xrightarrow{d_0} \cE_1 \otimes \kappa
  \end{equation}
  of morphisms in $\cat{Vect}_\sG(M)$, satisfying
  \begin{align*}
    d_0 \circ d_1 &= \phi \cdot \id_{\cE_1} \\
    (d_1 \otimes \kappa) \circ d_0 &= \phi \cdot \id_{\cE_0}.
  \end{align*}
  Treating these the same way as $2$-periodic complexes (even though
  they are not complexes), there is a dg-category of matrix
  factorizations, whose homotopy category we denote $\cat{MF}_\sG(M,
  \phi)$. Taking the Verdier quotient by totalizations of short exact
  sequences yields the {\it derived category of matrix factorizations}
  $\cat{DMF}_\sG(M, \phi)$; see \cite[Definition 3.9]{Ballard2014} for
  details. If $M$ is affine, then vector bundles on $M$ are projective
  objects and this quotient does nothing, i.e. $\cat{MF}_\sG(M, \phi)
  = \cat{DMF}_\sG(M, \phi)$.

  One can also define $\cat{MF}_\sG^{\cat{Coh}}(M, \phi)$ by
  considering pairs \eqref{eq:matrix-factorization} in
  $\cat{Coh}_\sG(M)$. Since $M$ is smooth, an adaptation of the proof
  that $K_\sG^\circ(M) \cong K_\sG(M)$ shows that the natural map
  $\cat{MF}_\sG(M, \phi) \xrightarrow{\sim}
  \cat{MF}_\sG^{\cat{Coh}}(M, \phi)$ is an equivalence
  \cite[Proposition 3.14]{Ballard2014}.
\end{definition}

\subsubsection{}

\begin{theorem}[{\cite[Theorem 3.9]{Orlov2004} \cite[Theorem 3.14]{Polishchuk2011}}] \label{thm:Dcrit-as-matrix-factorizations}
  Let $M$ be a smooth quasi-projective scheme acted on by a reductive
  group $\sG$. There is an equivalence of triangulated categories
  \begin{align*}
    \fC\colon \cat{DMF}_\sG(M, \phi) &\xrightarrow{\sim} D_\sG^{\crit}(M, \phi) \\
    (\cE_\bullet, d) &\mapsto \coker(\cE_1 \xrightarrow{d_1} \cE_0).
  \end{align*}
\end{theorem}

\begin{proof}[Proof sketch.]
  We only explain essential surjectivity when $M$ is affine, following
  \cite[Theorem 3.9]{Orlov2004}, which will suffice for the discussion
  in \S\ref{sec:Dcrit-functoriality}.

  Let $M_0 \coloneqq \phi^{-1}(0)$ for short, and
  $i\colon M_0 \hookrightarrow M$ be the embedding. Smoothness of $M$
  means $M_0$ is Gorenstein, and then one shows:
  \begin{itemize}
  \item every object in $D_\sG^{\crit}(M, \phi)$ is isomorphic to the
    image, under the projection map, of a (maximal Cohen--Macaulay)
    sheaf $\cF \in \cat{Coh}_\sG(M_0)$;
  \item the sheaf $i_*\cF \in \cat{Coh}_\sG(M)$ has a two-term
    resolution
    $0 \to \cE_1 \xrightarrow{d_1} \cE_0 \xrightarrow{f} i_*\cF \to 0$
    by vector bundles $\cE_i \in \cat{Vect}_\sG(M)$.
  \end{itemize}
  Since $\phi$ acts by zero on $i_*\cF$, there is an inclusion
  $d_0\colon \phi \cdot \cE_0 \hookrightarrow \ker(f) = \cE_1$. This
  completes $\cE_1 \xrightarrow{d_1} \cE_0$ into a matrix
  factorization.
\end{proof}

\subsubsection{}
\label{sec:Dcrit-functoriality}

We review some functors on $\cat{MF}_\sG(M, \phi)$. They induce
derived functors on $\cat{DMF}_\sG(M, \phi)$. See \cite[\S
  3]{Ballard2014} for details.
\begin{itemize}
\item Any $\sG$-equivariant morphism $f\colon M \to N$ induces a
  (functorial) pullback
  \begin{align*}
    f^*\colon \cat{MF}_\sG(N, \phi) &\to \cat{MF}_\sG(M, \phi\circ f) \\
    (\cE_\bullet, d) &\mapsto (f^*\cE_\bullet, f^*d)
  \end{align*}
  since pullback is exact on vector bundles.

\item Any proper $\sG$-equivariant morphism $f\colon M \to N$ induces
  a (functorial) pushforward
  \begin{align*}
    f_*\colon \cat{MF}_\sG(M, \phi \circ f) &\to \cat{MF}_\sG^{\cat{Coh}}(N, \phi) \cong \cat{MF}_\sG(N, \phi) \\
    (\cE_\bullet, d) &\mapsto (f_*\cE_\bullet, f_*d)
  \end{align*}
  since $f_*$ preserves coherence. To be clear, the notation
  $f_*\cE_\bullet$ here means to apply the non-derived functor
  $f_*\colon \cat{Coh}_\sG(M) \to \cat{Coh}_\sG(N)$ to each term in
  $\cE_\bullet$.

\item Given two potentials $\phi$ and $\psi$ on $M$, there is a tensor
  product
  \[ \otimes\colon \cat{MF}_\sG(M, \phi) \otimes \cat{MF}_\sG(M, \psi) \to \cat{MF}_\sG(M, \phi + \psi). \]
\end{itemize}
It is clear from the proof of
Theorem~\ref{thm:Dcrit-as-matrix-factorizations} that any reasonable
definition of these functors must be compatible with the pre-existing
ones in $D^{\crit}_\sG$ (or some enlargement like
$D^b\cat{Qcoh}_\sG/D^b\cat{Vect}_\sG$) under the equivalence $\fC$.
So, from here on, we stop distinguishing between $D^{\crit}_\sG$ and
$\cat{DMF}_\sG$ and freely switch between the two.

\subsubsection{}
\label{sec:dimensional-reduction}

\begin{theorem}[Dimensional reduction, \cite{Isik2013}] \label{thm:dimensional-reduction}
  Let $\pi\colon E \to X$ be a vector bundle on a smooth variety, and
  $Z \coloneqq s^{-1}(0)^{\text{derived}} \subset X$ be the derived zero
  locus of a section $s \in H^0(E)$. Then
  \[ D^b\cat{Coh}(Z) \simeq D^{\crit}_{\bC^\times}(E^\vee, \phi) \]
  where $\phi\colon E^\vee \to \bC$ is given by
  $\phi(x, f) \coloneqq f(s(x))$ for $x \in X$ and $f \in E^\vee_x$,
  and $\bC^\times$ acts by dilation on $E^\vee$.
\end{theorem}

For completeness, and also to facilitate the discussion in
\S\ref{sec:dimensional-reduction-observations}, we sketch Isik's
original proof of the theorem, written in terms of graded dg-algebras.
For a graded dg-algebra $\cA$, let $D^b_{\gr}\cat{Coh}(\cA)$ (resp.
$\cat{Perf}_{\gr}(\cA)$) be the bounded derived category of graded
coherent (resp. perfect) dg $\cA$-modules, and $D^{\crit}_{\gr}(\cA)
\coloneqq D^b_{\gr}\cat{Coh}(\cA)/\cat{Perf}_{\gr}(\cA)$.

\begin{proof}[Proof sketch.]
  Let $W \coloneqq \phi^{-1}(0)^{\text{derived}} \subset E^\vee$ for
  short. Let $\kappa$ denote the weight of the $\bC^\times$ action,
  which is equivalently a grading on $\cO_W$. So
  $D^{\crit}_{\bC^\times}(E^\vee, \phi) = D^{\crit}_{\gr}(\pi_*\cO_W)$
  by definition. If $E^\vee = \Spec \Sym \cE$, then $\pi_*\cO_W$ is
  quasi-isomorphic to ($\cE$ is in cohomological degree $0$)
  \[ \cB \coloneqq \Sym(0 \to \kappa \cO_X \xrightarrow{s} \cE \to 0) \]
  as sheaves of graded dg-algebras, by applying $\pi_*$ to $0 \to
  \kappa \cO_{E^\vee} \xrightarrow{s} \cO_{E^\vee} \to \cO_W \to 0$.
  By linear Koszul duality \cite{Mirkovic2010}, there is an
  equivalence
  \begin{equation} \label{eq:linear-koszul-duality}
    \begin{tikzcd}[column sep=large]
      D^b_{\gr}\cat{Coh}(\cB) \ar[shift left=2]{r}{\cF \mapsto \cA \otimes_{\cO_X} \cF^\vee} & D^b_{\gr}\cat{Coh}(\cA)^{\op} \ar[shift left=2]{l}{\cB \otimes_{\cO_X} \cG^\vee \mapsfrom \cG}
    \end{tikzcd}
  \end{equation}
  for the Koszul dual (with $\cE^\vee$ in cohomological degree $1$,
  and $t$ the Koszul dual of $\kappa$)
  \begin{align*}
    \cA &\coloneqq \Sym(0 \to \cE^\vee \xrightarrow{-s^\vee} t \cO_X \to 0) \\
        &= \wedge^\bullet \cE^\vee \otimes_{\cO_X} \cO_X[t] \cong \cO_Z[t],
  \end{align*}
  which is nothing more than the Koszul resolution of $\cO_Z[t]$. One
  checks easily that the equivalence identifies
  $\cat{Perf}_{\gr}(\cB) \simeq D^b_{\gr}\cat{Coh}(\cO_Z)^{\op}$. But
  \begin{equation} \label{eq:linear-koszul-dual-of-Dcrit}
    D^b_{\gr}\cat{Coh}(\cO_Z[t])/D^b_{\gr}\cat{Coh}(\cO_Z) \simeq D^b_{\gr}\cat{Coh}(\cO_Z[t^\pm]) = D^b\cat{Coh}(Z)
  \end{equation}
  where $\simeq$ is a sort of Quillen localization for
  $D^b_{\gr}\cat{Coh}$, and $=$ is tautological since graded
  $\cO_Z[t^\pm]$-modules are just $\cO_Z$-modules. The $\op$ in
  \eqref{eq:linear-koszul-duality} can be removed by applying the
  equivalence $R\!\cHom(-, \cO_Z)$.
\end{proof}

\subsubsection{}
\label{sec:dimensional-reduction-observations}

We make three important observations about the proof of
Theorem~\ref{thm:dimensional-reduction}, all of which are already
present in \cite{Toda2019}.

First, the $\bC^\times$-weight of the potential $\phi$ is the weight
$\kappa$ in Definition~\ref{def:critical-K-theory}, and for
Theorem~\ref{thm:dimensional-reduction} to hold, it is important that
$\kappa$ is non-trivial.

Second, since Koszul duality works $\sG$-equivariantly, everything in
Theorem~\ref{thm:dimensional-reduction} can be made $\sG$-equivariant
so long as the potential $\phi$ is $\sG$-invariant (but not
$\bC^\times$-invariant). Therefore the induced isomorphism $K_\sG(Z)
\cong K_{\sG \times \bC^\times}^{\crit}(E^\vee, \phi)$ is an
isomorphism of $\bk_\sG$-modules, not just of $\bZ$-modules.

Finally, the $\bC^\times$-equivariance was really only necessary for
the last equality in \eqref{eq:linear-koszul-dual-of-Dcrit}. Passing
to Grothendieck K-groups makes it entirely unnecessary, since
\[ K_{\sG \times \bC^\times}(\cO_Z[t^\pm]) = K_{\sG}(\cO_Z) \cong K_{\sG}(\cO_Z[t^\pm]). \]
The isomorphism comes from the long exact sequence $\cdots \to
K_{\sG}(\cO_Z) \xrightarrow{i_*} K_{\sG}(\cO_Z[t]) \to
K_{\sG}(\cO_Z[t^\pm]) \to 0$, where the map $i_*$ is in fact zero
since the coordinate $t$ has trivial $\sG$-weight, followed by the
Thom isomorphism $K_{\sG}(\cO_Z[t]) \cong K_{\sG}(\cO_Z)$. Neither of
these steps hold in $D^b\cat{Coh}$. (This was also observed in
\cite[Corollary 3.13]{Toda2023}.) Put differently, in critical
K-theory, we are allowed to specialize to $\kappa = 1$.

The conclusion is the {\it K-theoretic dimensional reduction}
statement that
\begin{equation} \label{eq:k-theoretic-dimensional-reduction}
  K_{\sG \times \bC^\times}^{\crit}(E^\vee, \phi) \cong K_{\sG}(Z) \cong K_{\sG}^{\crit}(E^\vee, \phi).
\end{equation}
By the discussion in \S\ref{sec:dg-stuff}, the derived zero locus $Z$
can be replaced here by the classical zero locus $Z^{\cl}$ with no
effect, which we freely do henceforth.

\subsubsection{}
\label{sec:dimensional-reduction-zero-potential}

A trivial case of dimensional reduction is when $E = X$ and $\phi = 0$
is identically zero:
\[ D^b\cat{Coh}_\sG(X) \simeq D_{\sG \times \bC^\times}^{\crit}(X, 0). \]
This equivalence is given by {\it totalization} on objects, i.e.
$(\cF^\bullet, d) \mapsto [\bigoplus_i \cF^{2i} \to \bigoplus_i
  \cF^{2i+1}]$ with maps in the matrix factorization given by $d$, and
the cohomological grading on the left hand side corresponds to the
grading by $\bC^\times$-weight on the right hand side. K-theoretic
dimensional reduction in this case says
\[ K_\sG(X) \cong K_\sG^{\crit}(X, 0). \]
In particular, $K_\sG^{\crit}(X, \phi)$ is a $K_\sG(X)$-module by
tensor product.

\subsubsection{}

Here is the prototypical example of critical K-theory and dimensional
reduction, a mild generalization of which is the {\it Kn\"orrer
  periodicity} $K^{\crit}_\sG(X \times \bC^2, \phi \boxplus xy) \cong
K^{\crit}_\sG(X, \phi)$.

\begin{example}
  Consider $\bC^2$, with coordinates $x$ and $y$, as the trivial line
  bundle $E^\vee$ over the $x$-axis $X \coloneqq \bC^1$. In the
  notation of Theorem~\ref{thm:dimensional-reduction}, let
  \[ s(x) = x, \qquad \phi(x, y) = xy. \]
  Let $\sT \coloneqq (\bC^\times)^2$ scale $x$ and $y$ with weights
  $t_1$ and $t_2$ respectively, so that $\kappa = t_1t_2$ is the
  $\sT$-weight of $\phi$. Set $\sA \coloneqq \ker \kappa \subset \sT$.

  In this setting, we can check K-theoretic dimensional reduction by
  computing the modules in
  \eqref{eq:k-theoretic-dimensional-reduction} explicitly. Let
  $Z \coloneqq \{x=0\} \subset X$ and
  $W \coloneqq \{xy=0\} \xhookrightarrow{i} \bC^2$. Then clearly
  \[ K_\sA(Z) = \bk_\sA. \]
  By considering the regular immersions $\{x=0\} \subset W$ and
  $\{y=0\} \subset W$, we claim
  \begin{equation} \label{eq:nodal-singularity-k-group}
    K_\sT(W) = \frac{\bk_\sT \cO_{\{x=0\}} \oplus \bk_\sT \cO_{\{y=0\}}}{\bk_\sT \cdot \left((1 - t_2) \cO_{\{x=0\}} - (1 - t_1) \cO_{\{y=0\}}\right)}.
  \end{equation}
  Indeed, a simple support argument shows $\cO_{\{x=0\}}$ and
  $\cO_{\{y=0\}}$ generate, and the relation is because both sides of
  the minus sign equal $\cO_0$. To show no other relations exist, use
  that $i_*\colon K_\sT(W) \to K_\sT(\bC^2)$ is injective since
  $i^*i_* = 1 - \kappa$ is a non-zerodivisor, and their images in
  $K_\sT(\bC^2) \cong \bk_\sT$ clearly satisfy no other relations.
  Finally, the only vector bundles on $W$ arise from $\cO_W$, which
  sits in the short exact sequence
  \[ 0 \to t_1 \cO_{\{y=0\}} \xrightarrow{x} \cO_W \to \cO_{\{x=0\}} \to 0. \]
  The result is that
  \[ K^{\crit}_\sT(\bC^2, xy) = \frac{K_\sT(W)}{\bk_\sT \cdot (t_1\cO_{\{y=0\}} - \cO_{\{x=0\}})} \cong \frac{\bk_\sT \cO_{\{y=0\}}}{\bk_\sT \cdot (1 - t_1t_2) \cO_{\{y=0\}}}. \]
  This is obviously isomorphic to $K^{\crit}_\sA(\bC^2, xy)$ as well
  as to $K_\sA(Z)$. Indeed, the linear Koszul duality
  \eqref{eq:linear-koszul-duality} identifies $\cO_0 \in K_\sA(Z)$
  with $\cO_{\{y=0\}} \in K_\sT(W)$.

  It is instructive to note, using
  \eqref{eq:nodal-singularity-k-group}, that the canonical map
  $K_\sT^\circ(W) \to K_\sT(W)$ is injective while $K_\sA^\circ(W) \to
  K_\sA(W)$ is not. Indeed, $\cO_W$ is torsion in $K_\sA(W)$:
  specializing to $(t_1, t_2) = (t, t^{-1})$, the relation in
  \eqref{eq:nodal-singularity-k-group} becomes
  \[ (1 - t^{-1})(\cO_{\{x=0\}} + t \cO_{\{y=0\}}) = (1 - t^{-1})\cO_W = 0. \]
\end{example}

\section{Braided multiplicative vertex coalgebras}
\label{sec:braided-coVA}

\subsection{General theory}
\label{sec:braided-coVA-theory}

\subsubsection{}

The main goal of this subsection is to define braided multiplicative
vertex coalgebras (Definition~\ref{def:vertex-coalgebra}). This will
be a synthesis of the multiplicative vertex algebras of \cite[\S
  3]{Liu2022} with the quantum vertex algebras of \cite{Etingof2000}
and with the vertex coalgebras of \cite{Hubbard2009}. Some judicious
notation and nomenclature originate from the latter.

In particular, the definition will be almost a categorical dual of the
notion of (non-equivariant, reduced) multiplicative vertex algebra in
\cite[\S 3]{Liu2022}. As with ordinary vertex algebras, see e.g.
\cite{Frenkel2004}, most of the complexity comes from a careful
treatment of the underlying (Laurent) series rings and modules.

\subsubsection{}

\begin{definition} \label{def:coVA-holomorphic-series}
  Let $R$ be a commutative ring and $V$ be an $R$-module. For a formal
  variable $z$, let
  \begin{equation} \label{eq:coVA-modules}
    V\pseries*{(1-z)^{-1}} \subset V\lseries*{(1-z)^{-1}} = V\pseries*{(1-z)^{-1}}[z]
  \end{equation}
  be the $R$-modules of $V$-valued formal power series and formal
  Laurent series in $(1-z)^{-1}$ respectively. We say an element of
  the latter is {\it holomorphic} if it lies in
  \begin{equation} \label{eq:coVA-holomorphic-series}
    V[z^\pm] \subset V\lseries*{(1-z)^{-1}},
  \end{equation}
  identified as an $R$-submodule via the binomial theorem
  \begin{equation} \label{eq:polynomial-subring}
    z^n = (1 - (1-z))^n = \sum_{k \ge 0} \binom{n}{k} (-1)^{n-k} (1 - z)^{n-k}.
  \end{equation}
  This also identifies the $R$-submodule $V[z^\pm] \subset
  V\lseries{(1-z^{-1})^{-1}}$ of holomorphic elements.
\end{definition}

If $V$ is actually an $R$-algebra, then all modules above also become
$R$-algebras.

\subsubsection{}

\begin{remark}
  Many objects in this subsection, morally, live on the multiplicative
  group $\bC^\times$ on which $z$ (or, later, $w$) is a coordinate,
  and will be analogues of pre-existing objects on the additive group
  $\bC$ whose coordinate we denote $u$ (or, later, $v$). Over $\bQ$,
  these variables are related by $z = \exp(u)$ and $w = \exp(v)$. For
  instance, under this identification,
  \begin{equation} \label{eq:series-rings-under-log}
    \bQ\pseries*{1-z} \cong \bQ\pseries*{u} = \bQ\pseries*{-u} \cong \bQ\pseries*{1-z^{-1}}
  \end{equation}
  since $1 - z = 1 - e^u = -u(1 + O(u))$ is a multiple of $u$ by a
  unit in $\bQ\pseries{u}$. Note that holomorphic elements
  \eqref{eq:coVA-holomorphic-series} have no poles in $z \in
  \bC^\times$, as the terminology suggests.
\end{remark}

\subsubsection{}

\begin{definition} \label{def:multiplicative-expansion}
  Let
  \begin{align}
    \iota_z\colon \bZ\lseries*{(1-zw)^{-1}} &\to \bZ[w^\pm]\lseries*{(1-z)^{-1}} \nonumber \\
    (1-zw)^n &\mapsto w^n \sum_{k \ge 0} (-1)^k \binom{n}{k} (1-w^{-1})^k (1-z)^{n-k} \label{eq:multiplicative-expansion}
  \end{align}
  denote the injective ring homomorphism which is uniquely
  characterized by the condition
  \[ (1 - zw) \iota_z (1 - zw)^{-1} = 1. \]
  The right hand side of \eqref{eq:multiplicative-expansion} can be
  viewed as an expansion of $((1-w) + w(1-z))^n$ using the binomial
  theorem. Since
  \[ \iota_z(zw) = 1 - \left((1 - w) + w(1 - z)\right) = w\left(1 - (1 - z)\right), \]
  clearly $\iota_z$ preserves the sub-ring $\bZ[(zw)^\pm]$ of
  holomorphic elements. Given an $R$-module $V$, we continue to use
  $\iota_z$ to denote the induced $R$-module homomorphism
  \[ \iota_z\colon V\lseries*{(1-zw)^{-1}} \to V[w^\pm]\lseries*{(1-z)^{-1}}. \]
  We refer to $\iota_z$ as {\it expansion} in the codomain
  $V[w^\pm]\lseries{(1-z)^{-1}}$. This name is because, analytically,
  it arises from series expansion in the domain $|1 - w^{-1}| < |1 -
  z|$.

  This is the multiplicative analogue of the ring homomorphism
  $\iota_u\colon \bZ\pseries{(u - v)^{-1}} \to \bZ[v]\pseries{u^{-1}}$
  given by series expansion in the domain $|u| > |v|$.
\end{definition}

\subsubsection{}

\begin{definition} \label{def:vertex-coalgebra}
  Let $R$ be a commutative ring. A {\it braided multiplicative vertex
    $R$-coalgebra} is the data of:
  \begin{enumerate}
  \item an $R$-module $V$ of {\it states} with a distinguished {\it
      covacuum} $\vac \in V^*$;
  \item a {\it translation operator} $D(z)\colon V \to V[z^\pm]$ that
    is {\it multiplicative}, i.e. $D(z)D(w) = D(zw)$;
  \item a {\it vertex coproduct} $\Y(z)\colon V \to (V \otimes
    V)\lseries*{(1-z)^{-1}}$;
  \item a {\it half-braiding operator} $C(z) \in \Hom(V \otimes V, (V
    \otimes V)\pseries*{(1-z)^{-1}})[z]$ (see
    \S\ref{sec:vertex-coalgebra-braiding-operator}).
  \end{enumerate}
  We write $(V, \vac, D, \Y, C)$ for short. This data must satisfy the
  following axioms for any $a \in V$:
  \begin{enumerate}
  \item (covacuum) letting $\cdots$ denote terms which vanish at
    $z=1$,
    \begin{alignat*}{3}
      (\vac \otimes \id) \Y(z)a &= a, \qquad \; (\id \otimes \vac)\Y(z)a &&= a + \cdots \in V[z^\pm], \\
      (\vac \otimes \id) C(z) &= \vac \otimes \id, \quad (\id \otimes \vac) C(z) &&= \id \otimes \vac;
    \end{alignat*}

  \item (skew symmetry) $C(z) \Y(z)a$ and $\sigma_{12} C(z^{-1})
    \Y(z^{-1}) D(z)a$ are holomorphic and are equal in $(V \otimes
    V)[z^\pm]$, where $\sigma_{ij}$ denote the map which swaps the
    $i$-th and $j$-th tensor factors;
    
  \item (weak coassociativity) $(\Y(z) \otimes \id) \Y(w)a \equiv (\id
    \otimes \Y(w)) \Y(zw)a$, where $\equiv$ means that both sides are
    expansions, in their respective domains, of the same element of
    \begin{equation} \label{eq:weak-coassociativity-parent}
      (V \otimes V \otimes V)\pseries*{(1 - z)^{-1}, (1 - w)^{-1}, (1 - zw)^{-1}}[z, w];
    \end{equation}

  \item (Yang--Baxter relations) $C(w) \otimes \id$ and $\id \otimes
    C(z)$ commute, and, for any $b \in V \otimes V$,
    \begin{align}
      \sigma_{12}(\id \otimes \Y(z)) C(zw)b &\equiv (\id \otimes C(zw))\sigma_{12}(C(w) \otimes \id)(\id \otimes \Y(z))b, \label{eq:yang-baxter-coproduct-1} \\
      \sigma_{23}(\Y(z) \otimes \id) C(w)b &\equiv (C(zw) \otimes \id)\sigma_{23}(\id \otimes C(w))(\Y(z) \otimes \id)b. \label{eq:yang-baxter-coproduct-2}
    \end{align}
  \end{enumerate}
  The vertex coalgebra is {\it holomorphic} if actually $\Y(z)a$ and
  $C(z)b$ belong to $(V \otimes V)[z^\pm]$, for all $a \in V$ and $b
  \in V \otimes V$.
\end{definition}

In what follows, the term {\it vertex (co)algebra} refers to our
braided and multiplicative version by default, and the original notion
of vertex (co)algebra is called {\it additive}.

\subsubsection{}

To be precise regarding weak associativity, first observe that
\begin{align*}
  (\Y(z) \otimes \id)\Y(w) a &\in (V \otimes V \otimes V)\lseries*{(1-z)^{-1}}\lseries*{(1-w)^{-1}}, \\
  (\id \otimes \Y(w))\Y(zw)a &\in (V \otimes V \otimes V)\lseries*{(1-w)^{-1}}\lseries*{(1-zw)^{-1}},
\end{align*}
so they are not immediately comparable. Weak associativity means to
compare them using the expansions (induced from
Definition~\ref{def:multiplicative-expansion})
\begin{align*}
  \iota_w\colon &V^{\otimes 3}\pseries*{(1-z)^{-1}, (1-w)^{-1}, (1-zw)^{-1}}[z, w] \\
  &\hookrightarrow V^{\otimes 3}\pseries*{(1-z)^{-1}}[z^\pm]\pseries*{(1-w)^{-1}}[w] = V^{\otimes 3}\lseries*{(1-z)^{-1}}\lseries*{(1-w)^{-1}}, \\
  \iota_{zw}\colon &V^{\otimes 3}\pseries*{(1-z)^{-1}, (1-w)^{-1}, (1-zw)^{-1}}[z, w]  \\
  &\hookrightarrow V^{\otimes 3}\pseries*{(1-w)^{-1}}[w^\pm]\pseries*{(1-zw)^{-1}}[zw] = V^{\otimes 3}\lseries*{(1-w)^{-1}}\lseries*{(1-zw)^{-1}}.
\end{align*}
This is completely analogous to what happens for additive vertex
algebras, where the relevant expansions are the ring embeddings
\[ \bZ\lseries{u}\lseries{v} \hookleftarrow \bZ\left[(u-v)^{-1}\right] \hookrightarrow \bZ\lseries{v}\lseries{u}. \]

\subsubsection{}
\label{sec:vertex-coalgebra-braiding-operator}

To be precise regarding the Yang--Baxter axiom, first observe that the
half-braiding operator $C$ can equivalently be viewed as an operator
\[ C(z)\colon V \otimes V \to (V \otimes V)\lseries*{(1-z)^{-1}} \]
with the finiteness condition that it has uniformly lower-bounded
valuation in $(1-z)^{-1}$, i.e.
\[ C(z)b \in (1-z)^N \cdot (V \otimes V)\pseries*{(1-z)^{-1}} \]
for some constant $N \in \bZ$ independent of $b \in V \otimes V$. This
finiteness condition ensures that compositions in the Yang--Baxter
axiom are well-defined. For instance,
\[ C(w)\colon (V \otimes V)\lseries*{(1-z)^{-1}} \to (V \otimes V)\pseries*{(1-z)^{-1}, (1-w)^{-1}}[z, w] \]
instead of taking values in the much larger module $(V \otimes
V)\lseries{(1-w)^{-1}}\lseries{(1-z)^{-1}}$. Hence the left and right
hand sides of \eqref{eq:yang-baxter-coproduct-1} are elements
\begin{align*}
  \sigma_{12}(\id \otimes \Y(z)) C(zw)b &\in V^{\otimes 3}\lseries*{(1-z)^{-1}}\lseries*{(1-zw)^{-1}}, \\
  (\id \otimes C(zw))\sigma_{12}(C(w) \otimes \id)(\id \otimes \Y(z))b &\in V^{\otimes 3}\pseries*{(1-z)^{-1}, (1-w)^{-1}, (1-zw)^{-1}}[z, w],
\end{align*}
and can therefore be compared by expanding $(1 - w)^n$ using
$\iota_{zw}$. Similarly the left and right hand sides of
\eqref{eq:yang-baxter-coproduct-2} can be compared by expanding
$(1-zw)^n$ using $\iota_w$.

\subsubsection{}

Here is some motivation for Definition~\ref{def:vertex-coalgebra},
particularly those aspects which are not obviously categorical duals
of some aspect of vertex algebras \cite[\S 3]{Liu2022} and not simply
multiplicative analogues of some aspect of additive vertex coalgebras
\cite{Hubbard2009}.

First, we explain the translation operator and the vertex coproduct.
Recall that for vertex algebras, the translation operator and vertex
product are homomorphisms
\[ D(z)\colon V \to V\pseries{1-z}, \qquad Y(-, z)\colon V \otimes V \to V\lseries{1-z} \]
where the target of $D(z)$ is the sub-module $V\pseries{1-z} \subset
V\lseries{1-z}$ of series ``holomorphic'' at $z=1$. In the additive
case, the vertex product $Y(-, u)$ takes values in $\lseries{u}$ while
the vertex coproduct $\Y(u)$ is its categorical dual and takes values
in $\lseries{u^{-1}}$. Hence, for our vertex coalgebras, the
translation operator and vertex coproduct must be homomorphisms
\[ D(z)\colon V \to V[z^\pm], \qquad \Y(z)\colon V \to (V \otimes V)\lseries*{(1-z)^{-1}}, \]
where the target of $D(z)$ is the sub-module $V[z^\pm] \subset
V\lseries{(1-z)^{-1}}$ that we identified in
Definition~\ref{def:coVA-holomorphic-series}, consisting of series
``holomorphic'' at $z=1$.

\subsubsection{}

The (half-)braiding operator is a new and necessary feature, not
present in vertex algebras or in the additive setting. Recall that for
vertex algebras, the skew-symmetry axiom is
\begin{equation} \label{eq:vertex-algebra-skew-symmetry-axiom}
  Y(a, z) b = D(z) Y(b, z^{-1}) a.
\end{equation}
This equality is valid because of the ring isomorphism
\begin{align}
  \bZ\pseries*{1-z} &\cong \bZ\pseries*{1-z^{-1}} \label{eq:vertex-algebra-skew-symmetry-rings} \\
  1 - z &\mapsto -z(1 - z^{-1}) = -(1 - (1 - z^{-1}))^{-1} (1 - z^{-1})\nonumber
\end{align}
given, for instance, by forgetting the intermediate steps in
\eqref{eq:series-rings-under-log}. Note that $1 - (1-z^{-1}) \in
\bZ\pseries{1-z^{-1}}$ is a unit, so its inverse is well-defined.

On the other hand, for vertex coalgebras, the difficulty is that the
categorical dual of \eqref{eq:vertex-algebra-skew-symmetry-axiom}
requires us to compare
\[ \Y(z) \in (V \otimes V)\lseries*{(1-z)^{-1}}, \qquad \Y(z^{-1}) D(z) \in (V \otimes V)\lseries*{(1-z^{-1})^{-1}}, \]
but, in contrast to the situation in
\eqref{eq:vertex-algebra-skew-symmetry-rings}, there is no analogous
isomorphism between the rings $\bZ\pseries{(1-z)^{-1}}$ and
$\bZ\pseries{(1-z^{-1})^{-1}}$, not even over $\bQ$: the desired
identification is
\[ \bZ\pseries{(1-z)^{-1}} \ni (1 - z)^{-1} \mapsto 1 - (1 - z^{-1})^{-1} \in \bZ\pseries{(1-z^{-1})^{-1}}, \]
but the right hand side is a unit while the left hand side is not.
Instead, the half-braidings $C(z)$ and $C(z^{-1})$ are used to map the
two sides into their common sub-module $(V \otimes V)[z^\pm]$, where
they may be compared.

To emphasize, unlike for vertex algebras, there appears to be no
canonical notion of ``unbraided'' vertex coalgebra.

\subsubsection{}

\begin{remark} \label{rem:quantum-VA}
  Various notions of braiding for additive vertex algebra have
  previously appeared in the literature, for instance
  \cite{Etingof2000}. Often, such vertex algebras are ``quantum'' in
  the sense that there is an extra grading by the quantum parameter
  $\hbar$ which must be included as part of the defining axioms, and
  there is a (typically non-cocommutative) braiding operator
  \begin{equation} \label{eq:R-matrix-form}
    S_\hbar(z)\colon V \otimes V \to V \otimes V \otimes R_\hbar\lseries*{\cdots},
  \end{equation}
  where $R_\hbar$ is some $R$-algebra containing $\hbar$ and $\cdots$
  depends on how one chooses to expand in the spectral parameter $z$.
  Being a braiding operator means $S_\hbar(z)$ must satisfy the
  Yang--Baxter equation
  \[ (S_{\hbar}(z) \otimes \id)(\id \otimes S_{\hbar}(zw))(S_{\hbar}(w) \otimes \id) = (\id \otimes S_{\hbar}(w))(S_{\hbar}(zw) \otimes \id)(\id \otimes S_{\hbar}(z)). \]

  We refrain from using the words ``quantum'' and ``R-matrix'' for the
  following reasons. In our setup, in light of the skew symmetry
  axiom, the braiding operator should correspond to
  \[ S(z) \coloneqq C(z)^{-1} \sigma_{12} C(z^{-1}). \]
  But $C(z)$ is not required to be invertible in any sense, nor does
  it necessarily involve a parameter $\hbar$. Furthermore, even if
  $C(z)$ were invertible, $C(z)^{-1}$ is a series in $(1-z)^{-1}$
  while $C(z^{-1})$ is a series in $(1-z^{-1})^{-1}$ and such a
  composition is typically not well-defined. Finally, asking for
  $C(z)$, and therefore $S(z)$, to be an operator of the form
  \eqref{eq:R-matrix-form} is a much stronger condition than what we
  imposed in Definition~\ref{def:vertex-coalgebra}, because
  \[ V \otimes R\lseries{\cdots} \subsetneq V\lseries{\cdots} \]
  is a proper submodule. In particular, the half-braiding operators
  constructed in \S\ref{sec:quiver-coVA} will not be of the form
  \eqref{eq:R-matrix-form}.
\end{remark}

\subsubsection{}

\begin{proposition}[cf. {\cite[Lemma 3.2.5]{Liu2022}}] \label{prop:translation-colocality}
  Let $(V, \vac, D, Y, C)$ be a vertex coalgebra. For all $a \in V$:
  \begin{enumerate}
  \item (translation) $\Y(z) D(w)a \equiv (\id \otimes D(w)) \Y(zw)a$;
  \item (colocality) $(\iota_z C(z/w) \otimes \id) (\id \otimes
    \Y(w))\Y(z)a \equiv \sigma_{12}(\iota_w C(w/z) \otimes \id)(\id \otimes
    \Y(z))\Y(w)a$.
  \end{enumerate}
\end{proposition}

\begin{proof}
  Applying $\id \otimes \vac$ to the skew symmetry axiom gives $(\id
  \otimes \vac)\Y(z)a = D(z)a$. Using this followed by weak
  coassociativity,
  \begin{align*}
    \Y(z) D(w)a
    &= (\id \otimes \id \otimes \vac)(\Y(z) \otimes \id) \Y(w)a \\
    &\equiv (\id \otimes \id \otimes \vac)(\id \otimes \Y(w))\Y(zw)a
    = (\id \otimes D(w))\Y(zw)a.
  \end{align*}
  Similarly, applying $\id \otimes \vac \otimes \id$ to weak
  coassociativity gives $(D(z) \otimes \id)\Y(w)a \equiv \Y(zw)a$,
  also called {\it translation covariance}. Using this, weak
  coassociativity and skew symmetry,
  \begin{align*}
    (\iota_z C(z/w) \otimes \id) (\id \otimes \Y(w))\Y(z)a
    &\equiv (C(z/w) \otimes \id) (\Y(z/w) \otimes \id)\Y(w)a \\
    &= \sigma_{12}(C(w/z) \otimes \id) (\Y(w/z) D(z/w) \otimes \id)\Y(w)a \\
    &\equiv \sigma_{12}(C(w/z) \otimes \id) (\Y(w/z) \otimes \id)\Y(z)a \\
    &\equiv \sigma_{12}(\iota_w C(w/z) \otimes \id) (\id \otimes \Y(z)) \Y(w)a.
  \end{align*}
  Note that weak coassociativity says both sides of the first $\equiv$
  are expansions of
  \[ (C(z/w) \otimes \id) f_a \in (V \otimes V \otimes V)\pseries*{(1-z)^{-1}, (1-w)^{-1}, (1-z/w)^{-1}}[z/w, w] \]
  for some element $f_a$ in the same module, and so $C(z/w)$ must also
  be expanded in the appropriate domains, whence the $\iota_z$ on the
  left hand side. The $\iota_w$ on the right hand side of the last
  $\equiv$ arises from similar considerations.
\end{proof}

\subsubsection{}

\begin{remark}
  If one assumes that the half-braiding operators are invertible, then
  translation and colocality, along with the covacuum and Yang--Baxter
  axioms, together imply skew symmetry and weak coassociativity. This
  is a converse of Proposition~\ref{prop:translation-colocality}.
  Therefore, skew symmetry and weak coassociativity may be replaced by
  translation and colocality, forming an alternate set of defining
  axioms for vertex coalgebras. We will not use this; some details can
  be found in \cite[Proposition 1.4]{Etingof2000}.
\end{remark}

\subsubsection{}
\label{sec:coVA-on-graded-module}

Later, $V = \bigoplus_{\alpha \in A} V(\alpha)$ will be graded by a
monoid $A$ such that $\#\{\alpha_1, \alpha_2 \in A : \alpha_1 +
\alpha_2 = \alpha\} < \infty$ for any $\alpha \in A$, and this grading
will be compatible with all the operators forming the vertex
coalgebra. Namely, the covacuum, translation operator, vertex
coproduct and half-braiding operator will split into components
\begin{align*}
  \vac_\alpha\colon V(\alpha) &\to R \\
  D_\alpha(z)\colon V(\alpha) &\to V(\alpha)[z^\pm] \\
  \Y_{\alpha,\beta}(z)\colon V(\alpha+\beta) &\to (V(\alpha) \otimes V(\beta))\lseries*{(1-z)^{-1}}, \\
  C_{\alpha,\beta}(z)\colon V(\alpha) \otimes V(\beta) &\to (V(\alpha) \otimes V(\beta))\lseries*{(1-z)^{-1}},
\end{align*}
and it suffices to write the vertex coalgebra axioms for each graded
piece. For instance, weak coassociativity is $(\Y_{\alpha,\beta}(z)
\otimes \id) \Y_{\alpha+\beta,\gamma}(w) \equiv (\id \otimes
\Y_{\beta,\gamma}(w)) \Y_{\alpha,\beta+\gamma}(zw)$ for all $\alpha,
\beta \in A$.

This grading is distinct from the usual grading (by conformal
dimension) on an additive vertex algebra, where different
$u$-coefficients of $Y_{n,m}(u)\colon V_n \otimes V_m \to
V\lseries{u}$ land in {\it different} graded pieces of $V =
\bigoplus_n V_n$.

\subsection{On various quiver moduli}
\label{sec:quiver-coVA}

\subsubsection{}

\begin{definition} \label{def:moduli-stacks}
  Let $Q$ be a quiver with vertices indexed by $i \in I$ and edges
  denoted by $e\colon i \to j$. For a dimension vector $\alpha =
  (\alpha_i)_i \in \bZ_{\ge 0}^{|I|}$, let
  \begin{align*}
    M_Q(\alpha) &\coloneqq \prod_{e\colon i \to j} \Hom(k^{\alpha_i}, k^{\alpha_j}) \\
    \GL(\alpha) &\coloneqq \prod_i \GL(\alpha_i), \qquad \lie{gl}(\alpha) \coloneqq \prod_i \End(\alpha_i)
  \end{align*}
  so that $\fM_Q(\alpha) \coloneqq [M_Q(\alpha)/\GL(\alpha)]$ is the
  moduli stack of representations of $Q$ of dimension $\alpha$. Write
  $\fM_Q \coloneqq \bigsqcup_\alpha \fM_Q(\alpha)$. Note that
  $\fM_Q(0) = \pt$.

  Given $Q$, let $Q^{\doub}$ be the associated {\it doubled} quiver,
  with the same vertex set but with a ``dual'' edge
  $e^*\colon j \to i$ added for each edge $i \to j$ in the original
  $Q$. Similarly, obtain the {\it tripled} quiver $Q^{\trip}$ from
  $Q^{\doub}$ by adding an extra loop $i \to i$ for each vertex
  $i \in I$. Then
  \begin{align*}
    \fM_{Q^{\doub}}(\alpha) &= [T^*M_Q(\alpha) / \GL(\alpha)] \\
    \fM_{Q^{\trip}}(\alpha) &= [T^*M_Q(\alpha) \times \lie{gl}(\alpha) / \GL(\alpha)].
  \end{align*}
  Let $x \in M_Q(\alpha)$, $x^* \in M_Q(\alpha)^*$, and
  $x^\circ \in \lie{gl}(\alpha)$ be coordinates.
  
  Since $Q$ is usually clear from context, we abbreviate
  $\fM^{\doub} \coloneqq \fM_{Q^{\doub}}$ and
  $\fM^{\trip} \coloneqq \fM_{Q^{\trip}}$ and omit writing the
  subscripts $Q$ in $\fM_Q$ and $M_Q$.
\end{definition}

\subsubsection{}

\begin{definition} \label{def:quiver-stacks-equivariance}
  Let
  \[ \sA \coloneqq (\bC^\times)^{\#\text{ edges}} \]
  act on $M(\alpha)$, and therefore on $\fM(\alpha)$, by scaling the
  linear maps corresponding to the edges of the quiver $Q$. The
  induced symplectic $\sA$-action on $T^*M(\alpha)$, and therefore on
  $\fM^{\doub}(\alpha)$, can be augmented by a $\bC^\times_{\hbar}$
  which scales the $M_Q(\alpha)^*$ directions, and therefore the
  symplectic form, with weight $\hbar$. Set
  \[ \sT \coloneqq \sA \times \bC^\times_{\hbar}. \]
  Finally, let $\bC^\times_{\hbar}$ scale the $\lie{gl}(\alpha)$
  directions in $\fM^{\trip}(\alpha)$ with weight $\hbar^{-1}$; this
  is necessary for the $\sT$-invariance of the potential
  \eqref{eq:Mtrip-potential} later.
\end{definition}

\subsubsection{}
\label{sec:quiver-moduli-K-theory}

\begin{definition}
  Set
  \[ K_\sT(\fM) \coloneqq \bigoplus_\alpha K_\sT(\fM(\alpha)) \]
  and similarly for localized, critical, etc. K-groups. Here we let
  the $\bC_\hbar^\times$ factor act trivially unless the quiver $Q$ is
  a doubled or tripled quiver. Let $\cV_{\alpha,i}$ be the
  tautological bundle of the $i$-th vertex in $\fM(\alpha)$, pulled
  back from $[\pt/\GL(\alpha_i)]$ along the obvious projection. We
  have
  \begin{equation} \label{eq:quiver-moduli-k-theory}
    K_\sT^\circ(\fM(\alpha)) \cong K_\sT(\fM(\alpha)) \cong K_{\sT \times \GL(\alpha)}(\pt) \eqqcolon \bk_\sT[s_{\alpha,i,j} : i \in I, \, 1 \le j \le \alpha_i]^{S(\alpha)},
  \end{equation}
  where $S(\alpha) \coloneqq \prod_{i \in I} S_{\alpha_i}$ with
  $S_{\alpha_i}$ acting by permutation on the variables
  $\{s_{\alpha,i,j}\}_j$. Each $s_{\alpha,i,j}$ represents a line
  bundle, and in K-theory $\cV_{\alpha,i} = \sum_j s_{\alpha,i,j}$. 
\end{definition}

Equivariant K-theory typically does not have a K\"unneth theorem, but
from \eqref{eq:quiver-moduli-k-theory}, clearly
\begin{equation} \label{eq:kunneth-property-moduli-stacks}
  \boxtimes\colon K_\sT(\fM(\alpha)) \otimes_{\bk_\sT} K_\sT(\fM(\beta)) \xrightarrow{\sim} K_\sT(\fM(\alpha) \times \fM(\beta))
\end{equation}
is an isomorphism of $\bk_\sT$-modules.

\subsubsection{}

\begin{definition}
  On $\fM(\alpha) \times \fM(\beta)$, let (with the first term in
  degree zero)
  \begin{equation} \label{eq:bilinear-bundle-moduli-stack}
    \cE_{\alpha,\beta} \coloneqq \bigg[\bigoplus_i \cV_{\alpha,i}^\vee \boxtimes \cV_{\beta,i} \xrightarrow{\Xi} \bigoplus_{i \to j} \cV_{\alpha,i}^\vee \boxtimes \cV_{\beta,j}\bigg]
  \end{equation}
  where, if $\xi_{\alpha, i \to j}\colon \cV_{\alpha,i} \to
  \cV_{\alpha,j}$ is the universal morphism of the edge $i \to j$,
  then
  \[ \Xi \coloneqq \bigoplus_{i \to j} \left(\id \boxtimes \xi_{\beta, i \to j} - \xi_{\alpha, i \to j}^* \boxtimes \id\right). \]
  This is the ``bilinear'' version of the tangent complex
  $\bT_{\fM(\alpha)}$. In particular, $\bT_{\fM(\alpha)} =
  \Delta^*\cE_{\alpha,\alpha}[1]$ for the diagonal embedding
  $\Delta\colon \fM(\alpha) \to \fM(\alpha) \times \fM(\alpha)$. Note
  that edges may carry non-trivial $\sT$-weights which we did not
  explicitly write in \eqref{eq:bilinear-bundle-moduli-stack}, cf. the
  explicit formula \eqref{eq:bilinear-bundle-Mtrip} for $\fM^{\trip}$.

  A slightly different geometric characterization of
  $\cE_{\alpha,\beta}$, more natural from the perspective of Hall
  algebras, is given in
  Lemma~\ref{lem:bilinear-bundle-as-relative-tangent}.
\end{definition}

\subsubsection{}
\label{sec:theta-expansion}

\begin{definition}
  Given a line bundle $\cL$ on a space $X$, define the formal series
  \[ \frac{1}{1 - z\cL} \coloneqq \cL^\vee \sum_{k \ge 0} (1 - z)^{-k-1} (1 - \cL^\vee)^k \in K_\sG^\circ(X)\lseries*{(1-z)^{-1}} \]
  cf. \eqref{eq:multiplicative-expansion}. It is an inverse to $1 -
  z\cL = \wedge_{-z}^\bullet(\cL)$ in its domain. Extend this
  multiplicatively to $K_\sG^\circ(X)$: if $\cE_1, \cE_2$ are
  $\sG$-equivariant vector bundles,
  \[ \wedge_{-z}^\bullet(\cE_1 - \cE_2) \coloneqq \wedge_{-z}^\bullet(\cE_1) \otimes \prod_\cL \frac{1}{1 - z\cL} \]
  where the product ranges over (K-theoretic) Chern roots $\cL$ of
  $\cE_2$. On $\fM(\alpha) \times \fM(\beta)$, define
  \[ \Theta_{\alpha,\beta}(z) \coloneqq \wedge^\bullet_{-z}(\cE_{\alpha,\beta}^\vee). \]
  Its inverse is clearly $\Theta_{\alpha,\beta}(z)^{-1} =
  \wedge^\bullet_{-z}(-\cE_{\alpha,\beta}^\vee)$.
\end{definition}

\subsubsection{}
\label{sec:monoid-and-action}

The stack $\fM$ is a monoid object with $[\pt/\bC^\times]$-action,
meaning that it admits:
\begin{itemize}
\item an associative {\it direct sum} map
  $\Phi_{\alpha,\beta}\colon \fM(\alpha) \times \fM(\beta) \to \fM(\alpha+\beta)$,
  given on points by $([x], [y]) \mapsto [x \oplus y]$ and on
  stabilizer groups by
  $(f, g) \mapsto \begin{psmallmatrix}f & 0\\0 & g\end{psmallmatrix}$;
  \item a compatible {\it scaling automorphism} map
    $\Psi_\alpha\colon [\pt/\bC^\times] \times \fM(\alpha) \to \fM(\alpha)$,
    given on points by the identity and on stabilizer groups by
    $(\lambda, f) \mapsto \lambda f$.
\end{itemize}
The torus $\sT$ acts trivially on $[\pt/\bC^\times]$. The action
$\Psi_\alpha$ induces the following grading on $K_\sT(\fM(\alpha))$.

\subsubsection{}

\begin{definition}
  Let $K([\pt/\bC^\times]) \eqqcolon \bZ[z^\pm]$. The {\it grading
    operator} associated to $\Psi_\alpha$ is
  \[ z^{\deg}\colon K_\sT(\fM(\alpha)) \xrightarrow{\Psi_\alpha^*} K_\sT([\pt/\bC^\times] \times \fM(\alpha)) \cong K_\sT(\fM(\alpha))[z^\pm]. \]
  Here, the identification $\cong$ is because the $\bC^\times$-action
  on $\pt \times \fM(\alpha) = \fM(\alpha)$ is trivial; it can also be
  viewed as a K\"unneth theorem for products with $[\pt/\bC^\times]$.
  
  On a product like $\fM(\alpha) \times \fM(\beta)$, let $\Psi_\alpha$
  act on only the $i$-th factor to get grading operators $z^{\deg_i}$.
  For instance, $z^{\deg} \cV_{\alpha,i} = z$ for any $\alpha$ and
  $i$, and so
  \[ z^{\deg_1} \cE_{\alpha,\beta} = z^{-1}, \qquad z^{\deg_2} \cE_{\alpha,\beta} = z. \]
  In what follows, we treat $z$ as a formal variable, forgetting its
  geometric origin as a line bundle on $[\pt/\bC^\times]$.
\end{definition}

\subsubsection{}

\begin{theorem} \label{thm:coVA-construction-general}
  $K_\sT(\fM)$ has a vertex $\bk_\sT$-coalgebra structure. In the
  notation of \S\ref{sec:coVA-on-graded-module}:
  \begin{enumerate}
  \item the covacuum is $\vac_0 = \id$ and $\vac_\alpha = 0$ for
    $\alpha \neq 0$;
  \item the translation operator is $D(z) \coloneqq z^{\deg}$;
  \item the vertex coproduct is
    \begin{equation} \label{eq:vertex-coproduct}
      \Y_{\alpha,\beta}(z) \coloneqq \Theta_{\alpha,\beta}(z) \otimes z^{\deg_1} \Phi_{\alpha,\beta}^*;
    \end{equation}
  \item the half-braiding operator $C_{\alpha,\beta}(z)$ is
    multiplication by $\Theta_{\alpha,\beta}(z)^{-1}$.
  \end{enumerate}
\end{theorem}

Ignoring the half-braiding operator, this is, almost verbatim, a
dualized (in the coalgebra sense) version of the construction
\cite[Theorem 3.3.5]{Liu2022} of a multiplicative vertex algebra
structure on the operational K-homology of moduli stacks, which itself
is based on the original constructions in \cite{Joyce2021,Gross2022}.
As such, most of the proof of the theorem is formally identical to a
dualized version of the original proof, and will occupy the remainder
of this subsection.

\subsubsection{}

\begin{remark} \label{rem:coVA-construction-general-braiding}
  In Remark~\ref{rem:quantum-VA}, we observed that the skew-symmetry
  axiom suggests the ill-defined ``braiding operator''
  \[ C_{\beta,\alpha}(z)^{-1} C_{\alpha,\beta}(z^{-1}) = \Theta_{\beta,\alpha}(z) \Theta_{\alpha,\beta}(z^{-1})^{-1}. \]
  Motivated by the identity of rational functions $1/(1-x) =
  -x^{-1}/(1-x^{-1})$, we may instead consider the well-defined
  operator
  \[ S_{\alpha,\beta}(z) \coloneqq (-z)^{\rank \cE_{\alpha,\beta}} \det(\cE_{\alpha,\beta}) \Theta_{\beta,\alpha}(z) (\Theta_{\alpha,\beta}(z)^\vee)^{-1}. \]
  We refer to $S_{\alpha,\beta}(z)$ as the {\it braiding operator}
  associated to the vertex coalgebra. It will play an important role
  in the main compatibility Theorems~\ref{thm:compatibility-general}
  and \ref{thm:compatibility-preprojective}.
\end{remark}

\subsubsection{}

\begin{proof}[Proof of Theorem~\ref{thm:coVA-construction-general}.]
To begin, we first observe that the pullback $\Phi_{\alpha,\beta}^*$
in the vertex coproduct \eqref{eq:vertex-coproduct} is well-defined.
This is because
\[ K_\sT(\fM(\alpha)) \cong K_\sT^\circ(\fM(\alpha)) \]
by smoothness of $M(\alpha)$, and arbitrary pullbacks exist for
$K_\sT^\circ$. Furthermore, the codomain of the pullback is correct
because of the K\"unneth property
\eqref{eq:kunneth-property-moduli-stacks}.

This may seem like a pedantic remark, but, in the similar construction
of \S\ref{sec:preprojective-coVA}, the existence of
$\Phi_{\alpha,\beta}^*$ will be the primary technical issue.

\subsubsection{}

Many of the vertex coalgebra axioms will follow almost formally from
corresponding properties of $\Theta_{\alpha,\beta}(z)$ which we
collect here. First, in K-theory, $\cE_{\alpha,0} = 0 =
\cE_{0,\alpha}$, coming from the formula
\eqref{eq:bilinear-bundle-moduli-stack}, which implies that
\begin{equation} \label{eq:theta-normalization}
  \Theta_{0,\alpha}(z) = \Theta_{\alpha,0}(z) = 1.
\end{equation}
Second, the formula \eqref{eq:bilinear-bundle-moduli-stack} for
$\cE_{\alpha,\beta}$ is bilinear and weight $\pm 1$ in its factors, in
the sense that
\begin{equation} \label{eq:bundle-bilinearity}
  \begin{alignedat}{3}
    (\Phi_{\alpha,\beta} \times \id)^*(\cE_{\alpha+\beta,\gamma}) &= \pi_{13}^*(\cE_{\alpha,\gamma}) \oplus \pi_{23}^*(\cE_{\beta,\gamma}) \qquad &&(\Psi_\alpha \times \id)^*(\cE_{\alpha,\beta}) &&= \pi_1^*(\cL^\vee) \otimes \pi_{23}^*(\cE_{\alpha,\beta}) \\
    (\id \times \Phi_{\beta,\gamma})^*(\cE_{\alpha,\beta+\gamma}) &= \pi_{12}^*(\cE_{\alpha,\beta}) \oplus \pi_{13}^*(\cE_{\alpha,\gamma}) \qquad &&(\id \times \Psi_\beta)^*(\cE_{\alpha,\beta}) &&= \pi_2^*(\cL) \otimes \pi_{13}^*(\cE_{\alpha,\beta})
  \end{alignedat}
\end{equation}
where $\pi_i$ and $\pi_{ij}$ are projections and
$\cL \in K([\pt/\bC^\times])$ is the weight-$1$ representation. Hence
\begin{equation} \label{eq:theta-bilinearity}
  \begin{alignedat}{3}
    (\Phi_{\alpha,\beta} \times \id)^*\Theta_{\alpha+\beta,\gamma}(z) &= \Theta_{\alpha,\gamma}(z) \otimes \Theta_{\beta,\gamma}(z) \qquad &&w^{\deg_1}\Theta_{\alpha,\beta}(z) &&= \iota_z \Theta_{\alpha,\beta}(zw) \\
    (\id \times \Phi_{\beta,\gamma})^*\Theta_{\alpha,\beta+\gamma}(z) &= \Theta_{\alpha,\beta}(z) \otimes \Theta_{\alpha,\gamma}(z) \qquad &&w^{\deg_2}\Theta_{\alpha,\beta}(z) &&= \iota_z \Theta_{\alpha,\beta}(z/w)
  \end{alignedat}
\end{equation}
using Lemma~\ref{lem:theta-deg} below. Here and henceforth we omit the
pullbacks $\pi_{ij}^*$ to avoid clutter.

\subsubsection{}

\begin{lemma} \label{lem:theta-deg}
  If $\cL$ is a line bundle such that $w^{\deg} \cL = w \cL$, then
  \[ w^{\deg} \frac{1}{1 - z\cL} = \iota_z \frac{1}{1 - zw \cL}. \]
\end{lemma}

\begin{proof}
  Since $(1-zw)^{-1}$ is equal to $\iota_z (1 - zw)^{-1}$ on a
  non-trivial analytic neighborhood,
  \begin{align*}
    \iota_z (1 - zw)^{-k-1}
    &= w^{-k} \frac{\partial_z^k}{k!} \iota_z (1 - zw)^{-1} \\
    &= \frac{w^{-k-1}}{k!} \sum_{j \ge 0} \frac{(j+k)!}{j!} (1 - z)^{-j-k-1} (1 - w^{-1})^j
  \end{align*}
  Plug this into $\iota_z (1 - zw\cL)^{-1}$ and apply the binomial
  theorem to conclude.
\end{proof}

\subsubsection{}

\begin{proposition}[Covacuum]
  $D(1) = \id$ and
  \begin{alignat*}{3}
    (\vac \otimes \id)\Y(z) &= \id, \qquad \; (\id \otimes \vac)\Y(z) &&= D(z), \\
    (\vac \otimes \id) C(z) &= \vac \otimes \id, \quad (\id \otimes \vac) C(z) &&= \id \otimes \vac.
  \end{alignat*}
\end{proposition}

\begin{proof}
  Since $\vac$ is only non-zero on $V(0)$, where it is the identity,
  it suffices to check the equations for $\Y_{0,\alpha}$ and
  $\Y_{\alpha,0}$, and $C_{0,\alpha}$ and $C_{\alpha,0}$. This is just
  an exercise in unrolling notation, using
  \eqref{eq:theta-normalization} and that $\Phi_{0,\alpha}^* =
  \Phi_{\alpha,0}^* = \id$.
\end{proof}

\subsubsection{}

\begin{proposition}[Skew symmetry]
  $C_{\alpha,\beta}(z) \Y_{\alpha,\beta}(z) = \sigma_{12}
  C_{\beta,\alpha}(z^{-1}) \Y_{\beta,\alpha}(z^{-1}) D(z)$.
\end{proposition}

\begin{proof}
  The left hand side is $z^{\deg_1}\Phi_{\alpha,\beta}^*$. Since
  $z^{\deg} \Phi^* = \Phi^* z^{\deg}$ and $z^{\deg} = z^{\deg_1}
  z^{\deg_2}$, the right hand side becomes $\sigma_{12} z^{\deg_2}
  \Phi_{\beta,\alpha}^*$. These are obviously equal.
\end{proof}

\subsubsection{}

\begin{proposition}[Weak coassociativity]
  \[ (\Y_{\alpha,\beta}(z) \otimes \id) \Y_{\alpha+\beta,\gamma}(w)a \equiv (\id \otimes \Y_{\beta,\gamma}(w)) \Y_{\alpha,\beta+\gamma}(zw)a. \]
\end{proposition}

\begin{proof}
  Using the first line of the bilinearity
  \eqref{eq:theta-bilinearity}, the left hand side becomes
  \begin{align*}
    &\Theta_{\alpha,\beta}(z) \otimes z^{\deg_1} (\Phi_{\alpha,\beta} \times \id)^* \left[\Theta_{\alpha+\beta,\gamma}(w) \otimes w^{\deg_1}\Phi^*_{\alpha+\beta,\gamma}a\right] \\
    &\equiv \left(\Theta_{\alpha,\beta}(z) \otimes \Theta_{\alpha,\gamma}(zw) \otimes \Theta_{\beta,\gamma}(w)\right) \otimes \left[z^{\deg_1} (\Phi_{\alpha,\beta} \times \id)^* w^{\deg_1} \Phi^*_{\alpha+\beta,\gamma}a\right].
  \end{align*}
  Similarly, using the second line of \eqref{eq:theta-bilinearity},
  the right hand side becomes
  \begin{align*}
    &\Theta_{\beta,\gamma}(w) \otimes w^{\deg_2} (\id \times \Phi_{\beta,\gamma})^* \left[\Theta_{\alpha,\beta+\gamma}(zw) \otimes (zw)^{\deg_1}\Phi^*_{\alpha,\beta+\gamma}a\right] \\
    &\equiv \left(\Theta_{\beta,\gamma}(w) \otimes \Theta_{\alpha,\beta}(z) \otimes \Theta_{\alpha,\gamma}(zw)\right) \otimes \left[w^{\deg_2} (\id \times \Phi_{\beta,\gamma})^* (zw)^{\deg_1}\Phi^*_{\alpha,\beta+\gamma}a\right].
  \end{align*}
  Finally, $z^{\deg_1}(\Phi \times \id)^* w^{\deg_1} = (zw)^{\deg_1}
  w^{\deg_2} (\Phi \times \id)^*$ while $(zw)^{\deg_1}$ commutes with
  $(\id \times \Phi)^*$. We are done by the associativity of $\Phi$.
\end{proof}

\subsubsection{}

\begin{proposition}[Yang--Baxter relations]
  Multiplication by $\Theta_{\alpha,\beta}(z)^{\pm 1}$ is an operator
  with uniformly lower-bounded valuation in $(1-z)^{-1}$ (see
  \S\ref{sec:vertex-coalgebra-braiding-operator}), and
  \begin{align*}
    \sigma_{12}(\id \otimes \Y_{\beta,\gamma}(z)) C_{\alpha,\beta+\gamma}(zw)b &\equiv (\id \otimes C_{\alpha,\gamma}(zw))\sigma_{12}(C_{\alpha,\beta}(w) \otimes \id)(\id \otimes \Y_{\beta,\gamma}(z))b, \\
    \sigma_{23}(\Y_{\alpha,\beta}(z) \otimes \id) C_{\alpha+\beta,\gamma}(w)b &\equiv (C_{\alpha,\gamma}(zw) \otimes \id)\sigma_{23}(\id \otimes C_{\beta,\gamma}(w))(\Y_{\alpha,\beta}(z) \otimes \id)b.
  \end{align*}
\end{proposition}

\begin{proof}
  The claim about the valuation follows because, by definition,
  $\Theta_{\alpha,\beta}(z)^{\pm 1}$ is a Laurent series in
  $(1-z)^{-1}$. For the Yang--Baxter relations, using the second line
  of the bilinearity \eqref{eq:theta-bilinearity}, the left hand side
  of the first equation becomes
  \begin{align*}
    &\Theta_{\beta,\gamma}(z) \otimes z^{\deg_2}(\id \times \Phi_{\beta,\gamma})^* \left(\Theta_{\alpha,\beta+\gamma}(zw)^{-1} \otimes b\right) \\
    &\equiv \Theta_{\alpha,\beta}(w)^{-1} \otimes \Theta_{\alpha,\gamma}(zw)^{-1} \otimes \Theta_{\beta,\gamma}(z) \otimes z^{\deg_2}(\id \times \Phi_{\beta,\gamma})^* b.
  \end{align*}
  This is manifestly equal to the right hand side. The second equation
  follows similarly.
\end{proof}

This concludes the proof of Theorem~\ref{thm:coVA-construction-general}.
\end{proof}

\subsubsection{}

\begin{remark} \label{rem:theta-choice}
  There is a good amount of freedom in the choice of
  $\Theta_{\alpha,\beta}(z)$; the proof only required the
  bilinearity properties \eqref{eq:bundle-bilinearity}. However, the
  choice given here is the unique one compatible with the K-theoretic
  Hall algebra structure, see \S\ref{sec:compatibility-KHA}.
\end{remark}

\subsection{On the preprojective stack}
\label{sec:preprojective-coVA}

\subsubsection{}

\begin{definition}
  The action of $\GL(\alpha)$ on $T^*M(\alpha)$ is Hamiltonian. Let
  $\mu_\alpha\colon T^*M(\alpha) \to \lie{gl}(\alpha)^*$ be its moment
  map. Explicitly, it is the sum of commutators
  \[ \mu_\alpha(x, x^*) = \sum_{e\colon i \to j} [x_e^*, x_e] \]
  where $x_e$ is the $e$-th component of $x$. Define the {\it
    preprojective stack}
  \begin{equation} \label{eq:preprojective-stack}
    T^*\fM(\alpha) \coloneqq \left[\mu_\alpha^{-1}(0)/\GL(\alpha)\right]
  \end{equation}
  as the cotangent bundle of $\fM$, or, equivalently, as the moduli
  stack of representations of the preprojective algebra of $Q$.

  Note that $T^*\fM$ is still a monoid object with
  $[\pt/\bC^\times]$-action, in the sense of
  \S\ref{sec:monoid-and-action}, but whether its equivariant K-group
  has a K\"unneth property is not immediately obvious.
\end{definition}

\subsubsection{}

\begin{remark} \label{rem:flatness-of-moment-map}
  Following Varagnolo and Vasserot \cite{Varagnolo2022}, the more
  correct object to consider is the ($0$-shifted symplectic
  \cite{Pecharich2012}) dg-stack
  \[ [\mu^{-1}_\alpha(0)^{\text{derived}}/\GL(\alpha)] \]
  where one takes the derived instead of the ordinary zero locus.
  Recall from Example~\ref{ex:derived-zero-locus} that $T^*\fM$ is the
  classical truncation of this dg-stack, and that the two are the same
  if and only if $\mu_\alpha$ is a regular section. The combinatorial
  characterization \cite[Theorem 1.1]{Crawley-Boevey2001} of this
  condition fails in most examples of interest. Nonetheless, their
  equivariant K-groups are equal, see \S\ref{sec:dg-stuff}.
\end{remark}

\subsubsection{}

\begin{theorem} \label{thm:coVA-construction-preprojective}
  There is a vertex $\bk_{\sT,\loc}$-coalgebra structure on
  $K_\sT(T^*\fM_Q)_{\loc}$. In the notation of
  \S\ref{sec:coVA-on-graded-module}:
  \begin{enumerate}
  \item the covacuum is $\vac_0 = \id$ and $\vac_\alpha = 0$ for
    $\alpha \neq 0$;
  \item the translation operator is $D(z) \coloneqq z^{\deg}$;
  \item the vertex coproduct is
    \[ \Y_{\alpha,\beta}(z) \coloneqq \Theta_{\alpha,\beta}^{\trip}(z) \otimes z^{\deg_1} \Phi_{\alpha,\beta}^* \]
    where
    \[ \Theta^{\trip}_{\alpha,\beta}(z) \coloneqq \wedge^\bullet_{-z}(\cE_{\alpha,\beta}^{\trip,\vee}) \]
    is defined using the bilinear element $\cE_{\alpha,\beta}^{\trip}$
    for the tripled quiver $Q^{\trip}$;
  \item the half-braiding operator $C_{\alpha,\beta}(z)$ is
    multiplication by $\Theta_{\alpha,\beta}^{\trip}(z)^{-1}$.
  \end{enumerate}
\end{theorem}

This is the analogue of Theorem~\ref{thm:coVA-construction-general}
for $T^*\fM$. Like in Remark~\ref{rem:theta-choice}, the definition of
$\Theta_{\alpha,\beta}^{\trip}(z)$ here is the unique one compatible
with the Hall algebras of \S\ref{sec:compatibility-KHA}.

\subsubsection{}
\label{sec:coVA-preprojective-issues}

There are two issues which need to be addressed, in the remainder of
this subsection, after which the proofs of
Theorems~\ref{thm:coVA-construction-general} and
\ref{thm:coVA-construction-preprojective} are formally identical.
\begin{enumerate}
\item The K\"unneth property \eqref{eq:kunneth-property-moduli-stacks}
  is no longer completely clear. The most obvious way to obtain it
  (Lemma~\ref{lem:kunneth-property-preprojective-stack}) requires base
  change to $\bk_{\sT,\loc}$. Without this localization, it is unclear
  whether there is still a K\"unneth isomorphism; partial results in
  this direction are recorded in
  Appendix~\ref{sec:kunneth-properties}.

\item More severely, since $T^*\fM$ is in general singular,
  $K_\sT(T^*\fM(\alpha)) \neq K_\sT^\circ(T^*\fM(\alpha))$ and the
  pullback $\Phi_{\alpha,\beta}^*$ is not obviously well-defined. The
  solution
  (\S\ref{sec:preprojective-vertex-coproduct-from-critical-K-theory})
  is to realize $K_\sT(T^*\fM(\alpha))$ as a critical K-group by
  dimensional reduction
  (Lemma~\ref{lem:Mtrip-dimensional-reduction}), and critical
  K-groups have pullbacks along arbitrary morphisms
  (\S\ref{sec:Dcrit-functoriality}).
\end{enumerate}

\subsubsection{}

\begin{definition} \label{def:endomorphism-stacks}
  Let $\fN(\alpha) \coloneqq [\nu^{-1}(0)/G(\alpha)]$
  where
  \begin{align*}
    \nu\colon M(\alpha) \times \lie{gl}(\alpha) &\to M(\alpha) \\
    (x, x^\circ) &\mapsto \sum_{e\colon i \to j} \left(x_e x^\circ_i - x^\circ_j x_e\right)
  \end{align*}
  where $x_e$ and $x_i^\circ$ are the $e$-th and $i$-th component of
  $x$ and $x^\circ$ respectively. In other words, $\fN(\alpha)$ is the
  moduli stack of $(x, x^\circ)$ where $x \in \fM(\alpha)$ and
  $x^\circ$ is an endomorphism of $x$. Write
  \[ [N^{\nil}(\alpha) / \GL(\alpha)] \eqqcolon \fN^{\nil}(\alpha) \subset \fN(\alpha) \coloneqq [N(\alpha) / \GL(\alpha)] \]
  where $\fN^{\nil}(\alpha)$ is the closed substack where $x^\circ$ is
  nilpotent.
\end{definition}

\subsubsection{}

\begin{lemma} \label{lem:Mtrip-dimensional-reduction}
  Consider $\fM^{\trip}(\alpha)$ with the (clearly
  $\GL(\alpha)$-invariant) potential
  \begin{equation} \label{eq:Mtrip-potential}
    \phi_\alpha(x, x^*, x^\circ) \coloneqq \sum_i \sum_e \tr\left(x_e x_e^* x^\circ_i - x_e^* x_e x^\circ_i\right).
  \end{equation}
  Then, by K-theoretic dimensional reduction
  \eqref{eq:k-theoretic-dimensional-reduction},
  \begin{equation} \label{eq:Mtrip-dimensional-reduction}
    K_\sT(T^*\fM(\alpha)) \cong K_\sT^{\crit}(\fM^{\trip}(\alpha), \phi_\alpha) \cong K_\sT(\fN(\alpha))
  \end{equation}
  as $\bk_{\sT \times \GL(\alpha)}$-modules.
\end{lemma}

Definition~\ref{def:quiver-stacks-equivariance} for the action of
$\bC^\times_{\hbar} \subset \sT$ on $\fM^{\trip}$ was made precisely
so that $\phi_\alpha$ is $\sT$-invariant.

\begin{proof}
  Clearly $\phi_\alpha$ is linear in each of $x$, $x^*$, and
  $x^\circ$. So K-theoretic dimensional reduction may be applied in
  two different ways:
  \begin{itemize}
  \item to the $\lie{gl}(\alpha)$-bundle $\fM^{\trip}(\alpha) \to
    \fM^{\doub}(\alpha)$, which has fiber coordinate $x^\circ$,
    viewing
    \[ \phi_\alpha(x, x^*, x^\circ) = \sum_i \sum_e \tr\left([x_e, x^*_e] x^\circ_i\right); \]
  \item to the $M(\alpha)^*$-bundle $\fM^{\trip}(\alpha) \to
    \fN(\alpha)$, which has fiber coordinate $x^*$, viewing
    \[ \phi_\alpha(x, x^*, x^\circ) = \sum_i \sum_e \tr\left([x_e, x^\circ_i] x_e^* \right). \]
  \end{itemize}
  The results are the first and second isomorphisms in
  \eqref{eq:Mtrip-dimensional-reduction} respectively.
\end{proof}

\subsubsection{}
\label{sec:preprojective-vertex-coproduct-from-critical-K-theory}

It is clear that the direct sum map $\Phi_{\alpha,\beta}\colon
\fM^{\trip}(\alpha) \times \fM^{\trip}(\beta) \to
\fM^{\trip}(\alpha+\beta)$ on $\fM^{\trip}$ satisfies
$\phi_{\alpha+\beta} \circ \Phi_{\alpha,\beta} = \phi_\alpha \boxplus
\phi_\beta$. We can therefore use
\[ \begin{tikzcd}
    K_\sT^{\crit}(\fM^{\trip}(\alpha+\beta), \phi_{\alpha+\beta}) \ar{r}{\sim}[swap]{\text{dim. red.}} \ar{d}{\Phi_{\alpha,\beta}^*} & K_\sT(T^*\fM(\alpha+\beta)) \ar[dashed]{d} \\
    K_\sT^{\crit}(\fM^{\trip}(\alpha) \times \fM^{\trip}(\beta), \phi_\alpha \boxplus \phi_\beta) \ar{r}{\sim}[swap]{\text{dim. red.}} & K_\sT(T^*\fM(\alpha) \times T^*\fM(\beta)) 
  \end{tikzcd} \]
to define the dashed arrow, which we still denote
$\Phi_{\alpha,\beta}^*$ in a mild abuse of notation. Because
\[ \cE_{\alpha,\beta}^{\trip} \in K_\sT(\fM_\alpha^{\trip} \times \fM_\beta^{\trip}) \cong \bk_{\sT \times \GL(\alpha) \times \GL(\beta)}, \]
and K-theoretic dimensional reduction is linear with respect to this
ring, it is compatible with multiplication by
$\Theta_{\alpha,\beta}^{\trip}(z)$, and all the necessary bilinearity
properties \eqref{eq:theta-bilinearity} of $\Theta(z)$ are preserved.

\subsubsection{}
\label{sec:kunneth-property-preprojective-stack}

\begin{lemma} \label{lem:kunneth-property-preprojective-stack}
  The external tensor product
  \[ \boxtimes\colon K_\sT(T^*\fM(\alpha))_{\loc} \otimes_{\bk_{\sT,\loc}} K_\sT(T^*\fM(\beta))_{\loc} \to K_\sT(T^*\fM(\alpha) \times T^*\fM(\beta))_{\loc} \]
  is an isomorphism.
\end{lemma}

\begin{proof}
  This follows from \cite[Lemma 2.4.1]{Varagnolo2022}. We sketch a
  slight modification of their main idea, for the reader's
  convenience. By Lemma~\ref{lem:Mtrip-dimensional-reduction},
  $T^*\fM$ may be replaced by $\fN$. This stack has the advantage that
  \[ \boxtimes\colon K_\sT(\fN^{\nil}(\alpha)) \otimes_{\bk_\sT} K_\sT(\fN^{\nil}(\beta)) \to K_\sT(\fN^{\nil}(\alpha) \times \fN^{\nil}(\beta)) \]
  is an isomorphism (see Appendix~\ref{sec:kunneth-properties}). We
  claim that all $\bC^\times_{\hbar}$-fixed points in $\fN$ lie within
  $\fN^{\nil}(\alpha)$. This is because any such fixed point $(x,
  x^\circ)$ must, by definition, have an associated $1$-parameter
  subgroup $g(\lambda)\colon \bC_\hbar^\times \to \GL(\alpha)$ such
  that
  \[ (x, \lambda x^\circ) = (g(\lambda) x g(\lambda)^{-1}, g(\lambda) x^\circ g(\lambda)^{-1}). \]
  In particular, $\lambda x^\circ_i = g(\lambda)_i x^\circ_i
  g(\lambda)_i^{-1}$ where $g(\lambda)_i$ is the $i$-th component of
  $g(\lambda)$. When $\lambda \neq 1$, this is only possible if
  $x^\circ_i$ is nilpotent. Hence \cite{Aranha2024a} all (higher)
  $\sT$-equivariant K-theory groups of $\fN(\alpha) \setminus
  \fN^{\nil}(\alpha)$ are torsion and so
  \[ K_\sT(\fN^{\nil}(\alpha))_{\loc} \cong K_\sT(\fN(\alpha))_{\loc}. \qedhere \]
\end{proof}

\subsubsection{}

\begin{remark} \label{rem:coVA-general}
  More generally, one can take a quiver $Q$ with potential $W \in \bC
  Q/[\bC Q, \bC Q]$ and try to make $\bigoplus_\alpha
  K^{\crit}_\sT(\fM_Q(\alpha), \tr W_\alpha)$ into a vertex coalgebra
  following the exact same recipe as in
  Theorem~\ref{thm:coVA-construction-preprojective}. This works as
  long as there is a K\"unneth isomorphism
  \begin{align*}
    \boxtimes\colon &K_\sT^{\crit}(\fM_Q(\alpha), \tr W_\alpha) \otimes_{\bk_\sT} K_\sT^{\crit}(\fM_Q(\beta), \tr W_\beta) \\
                    &\xrightarrow{\sim} K_\sT^{\crit}(\fM_Q(\alpha) \times \fM_Q(\beta), \tr W_\alpha \boxplus \tr W_\beta).
  \end{align*}
  Non-equivariantly, i.e. with $\bZ$- instead of $\bk_\sT$-modules,
  this K\"unneth property always holds at the level of the singularity
  categories $D^{\crit}$, which is a Thom--Sebastiani-type theorem
  \cite[Theorem 5.15]{Ballard2014}. However, for various reasons, it
  does not always remain an isomorphism after passing to $K_0(-)$. For
  $Q^{\trip}$ in particular, we sidestepped this issue in
  \S\ref{sec:kunneth-property-preprojective-stack} by localization.

  However, we emphasize that the lack of a K\"unneth isomorphism is
  morally unimportant. Indeed, the proof of
  Theorem~\ref{thm:coVA-construction-general} works fine using
  \[ \Y_{\alpha,\beta}(z)\colon K_\sT^{\crit}(\fM_Q(\alpha+\beta), \tr W_{\alpha+\beta}) \to K_\sT^{\crit}(\fM_Q(\alpha) \times \fM_Q(\beta), \tr W_\alpha \boxplus \tr W_\beta)\lseries*{(1-z)^{-1}} \]
  and similarly for $C_{\alpha,\beta}(z)$, with some minor adjustments
  to notation. Then base change to localized K-groups is no longer
  necessary. The only technical caveat is that this is not a coproduct
  in the traditional sense of a map $V \to V \otimes V$.

  More importantly, the bilinear element $\cE_{\alpha,\beta}$ must be
  the one for $\fM_Q$ for the compatibility results of
  \S\ref{sec:compatibility-KHA} to hold. Note that although it
  consists of vector bundles, it is treated as an element of $\bk_{\sT
    \times \GL(\alpha) \times \GL(\beta)}$, and so in the $\bk_{\sT
    \times \GL(\alpha) \times \GL(\beta)}$-module
  $K_\sT^{\crit}(\fM_Q(\alpha) \times \fM_Q(\beta), \tr W_\alpha
  \boxplus \tr W_\beta)$, multiplication by $\cE_{\alpha,\beta}$ is
  non-zero in general.
\end{remark}

\section{The preprojective vertex bialgebra}
\label{sec:compatibility-KHA}

\subsection{Some Hall algebras}
\label{sec:HAs}

\subsubsection{}

\begin{definition}
  Let $\fM = \bigoplus_\alpha \fM(\alpha)$ be a moduli stack of
  objects in some abelian category. There is an associated {\it Ext
    stack}
  \[ \fM(\alpha, \beta) \coloneqq \{A \hookrightarrow B \twoheadrightarrow C\} \subset \fM(\alpha) \times \fM(\alpha+\beta) \times \fM(\beta) \]
  parameterizing short exact sequences, with natural projections
  \begin{equation} \label{eq:ext-stack-projections}
    \fM(\alpha) \times \fM(\beta) \xleftarrow{q_{\alpha,\beta}} \fM(\alpha,\beta) \xrightarrow{p_{\alpha,\beta}} \fM(\alpha+\beta)
  \end{equation}
  We often omit the subscripts on $p$ and $q$ when they are irrelevant
  or unambiguous. Suppose a torus $\sT$ acts on $\fM$, and there are
  well-defined maps $q^*$ and $p_*$ on $\sT$-equivariant K-groups such
  that $q^*$ is an isomorphism. Then there is an associative {\it Hall
    product}
  \[ \star\colon K_\sT(\fM(\alpha)) \otimes_{\bk_\sT} K_\sT(\fM(\beta)) \xrightarrow{\boxtimes} K_\sT(\fM(\alpha) \times \fM(\beta)) \xrightarrow{q^*} K_\sT(\fM(\alpha,\beta)) \xrightarrow{p_*} K_\sT(\fM(\alpha+\beta)) \]
  making $\bigoplus_\alpha K_\sT(\fM(\alpha))$ into a {\it K-theoretic
    Hall algebra} ({\it KHA}).

  This general sort of construction, and a broadly-applicable proof of
  its associativity, originates from the cohomological Hall algebras
  of \cite{Kontsevich2011}.
\end{definition}

\subsubsection{}

Following this general recipe, we now review the constructions of
three (successively more complicated) KHAs and compatibilities between
them.

\subsubsection{}

\begin{example}[Quiver KHA] \label{ex:quiver-KHA}
  Let $\fM = \bigsqcup_\alpha [M(\alpha)/\GL(\alpha)]$ be the moduli
  of quiver representations of a quiver $Q$
  (Definition~\ref{def:moduli-stacks}). Components of its Ext stack
  have the explicit presentation
  \[ \fM(\alpha,\beta) = [M(\alpha,\beta) / P(\alpha,\beta)], \]
  where $M(\alpha,\beta) \subset M(\alpha+\beta)$ is the vector
  subspace with non-negative weight with respect to the weight-$1$
  diagonal cocharacter
  $\bC^\times \to \GL(\alpha) \subset \GL(\alpha+\beta)$, and
  $P(\alpha,\beta) \subset \GL(\alpha+\beta)$ is the parabolic
  subgroup preserving $M(\alpha,\beta)$.
  \begin{itemize}
  \item The projection
    $q\colon \fM(\alpha,\beta) \to \fM(\alpha) \times \fM(\beta)$
    factors as
    \[ q\colon [M(\alpha,\beta)/P(\alpha,\beta)] \xrightarrow{\bar q} [M(\alpha) \times M(\beta)/P(\alpha,\beta)] \xrightarrow{r} [M(\alpha) \times M(\beta)/\GL(\alpha) \times \GL(\beta)] \]
    where $\bar q$ is an $\Ext^1$-bundle, so $\bar q^*$ is an
    isomorphism, and $r^*$ is an isomorphism on K-theory by
    \eqref{eq:equivariant-reduction}. The unipotent part of
    $P(\alpha,\beta)$ acts trivially on $M(\alpha) \times M(\beta)$ by
    definition.
  \item The projection $p\colon \fM(\alpha,\beta) \to
    \fM(\alpha+\beta)$ factors as
    \begin{equation} \label{eq:HA-p}
      p\colon [M(\alpha,\beta)/P(\alpha,\beta)] \xhookrightarrow{i} [M(\alpha+\beta)/P(\alpha,\beta)] \xrightarrow{\pi} [M(\alpha+\beta)/\GL(\alpha+\beta)]
    \end{equation}
    where $i$ is a closed immersion and $\pi$ is a proper projection.
    The latter is modeled on $[\pt/P] \to [\pt/\GL]$ which is nothing
    more than the projection $\GL/P \to \pt$ from a partial flag
    variety.
  \end{itemize}
  Hence $K_\sT(\fM) \coloneqq \bigoplus_\alpha K_\sT(\fM(\alpha))$
  becomes a KHA. Using that $K_\sT(\fM(\alpha)) \cong \bk_{\sT \times
    \GL(\alpha)}$ is just a Laurent polynomial ring, the Hall product
  $\star$ here has an explicit formula in the form of a {\it shuffle
    product}, see \S\ref{sec:quiver-KHA-shuffle-formula}.
\end{example}

\subsubsection{}

\begin{lemma} \label{lem:bilinear-bundle-as-relative-tangent}
  Let $\bT_{p_{\alpha,\beta}}$ denote the relative tangent complex of
  $p_{\alpha,\beta}$. Then
  \[ \bT_{p_{\alpha,\beta}} = q^*\cE_{\alpha,\beta} \in D^b\cat{Coh}_\sT(\fM(\alpha,\beta)) \]
\end{lemma}

This provides an alternative geometric meaning to our choice of
bilinear element $\cE_{\alpha,\beta}$
(Definition~\ref{eq:bilinear-bundle-moduli-stack}), and is crucial to
the compatibility (Theorems~\ref{thm:compatibility-general} and
\ref{thm:compatibility-preprojective}) of the Hall product with the
vertex coproduct.

\begin{proof}
  We only need this lemma in K-theory, so we only provide the proof in
  K-theory. The general proof follows the same idea but with more
  bookkeeping.

  Recall that a quotient stack $[X/G]$ has tangent complex
  $\bT_{[X/G]} = [\lie{g} \otimes \cO_X \to \cT_X]$, where $\cT_X$ is
  the tangent sheaf of $X$ (sitting in degree zero) and $\lie{g}$ is
  the Lie algebra of $G$. By the definition of relative tangent
  complexes, in K-theory we have
  \[ \bT_{p_{\alpha,\beta}} = \left(\cT_{M(\alpha,\beta)} - \lie{p}(\alpha,\beta) \otimes \cO_{M(\alpha,\beta)}\right) - \left(p^*\cT_{M(\alpha+\beta)} - \lie{gl}(\alpha+\beta) \otimes \cO_{M(\alpha,\beta)}\right) \]
  where $\lie{p}(\alpha+\beta)$ and $\lie{gl}(\alpha,\beta)$ are the
  Lie algebras of $P(\alpha+\beta)$ and $\GL(\alpha,\beta)$
  respectively. Also,
  \begin{align*}
    \iota^*\cT_{M(\alpha+\beta)} - \cT_{M(\alpha,\beta)} &= \sum_{i \to j} \cV_{\alpha,i}^\vee \boxtimes \cV_{\beta,j}, \\
    (\lie{gl}(\alpha+\beta) - \lie{p}(\alpha,\beta)) \otimes \cO &= \sum_i \cV_{\alpha,i}^\vee \boxtimes \cV_{\beta,i}.
  \end{align*}
  These are the parts of $p^*\bT_{\fM(\alpha+\beta)}$ with negative
  weight with respect to the weight-$1$ diagonal cocharacter
  $\bC^\times \to \GL(\alpha) \subset \GL(\alpha+\beta)$. Comparing
  with \eqref{eq:bilinear-bundle-moduli-stack}, we are done.
\end{proof}

\subsubsection{}

\begin{example}[Preprojective KHA, \cite{Varagnolo2022}]
  Let $T^*\fM = [Z/\GL]$ be the preprojective stack of a quiver $Q$
  and let $T^*\fM(\alpha,\beta) = [Z(\alpha,\beta)/P(\alpha,\beta)]$
  be its corresponding Ext stack. Explicitly, it fits into the
  commutative diagram
  \[ \begin{tikzcd}
      {[Z(\alpha) \times Z(\beta)/P(\alpha,\beta)]} \ar[hookrightarrow]{d} & T^*\fM(\alpha,\beta) \ar{l}[swap]{\tilde q_Z} \ar{r}{p_Z} \ar[hookrightarrow]{d} & T^*\fM(\alpha+\beta) \ar[hookrightarrow]{d} \\
      \left[\dfrac{\lie{p}^\perp \times T^*M(\alpha) \times T^*M(\beta)}{P(\alpha,\beta)}\right] & \fM^{\doub}(\alpha,\beta) \ar{l}[swap]{\tilde q} \ar{r}{p} & \fM^{\doub}(\alpha+\beta)
    \end{tikzcd} \]
  where $\lie{p} \subset \lie{gl}(\alpha+\beta)$ is the Lie algebra of
  $P(\alpha,\beta)$, and
  $\tilde q(A \subset B) \coloneqq (\mu_{\alpha+\beta}(A,B/A), A, B/A)$.
  Both squares are Cartesian; this would be false without the
  $\lie{p}^\perp$ factor in the bottom left.

  We know $p$ is proper from Example~\ref{ex:quiver-KHA}, and
  $\tilde q$ is lci since both its source and target are smooth. The
  vertical inclusions are badly-behaved in general, see
  Remark~\ref{rem:flatness-of-moment-map}, so while $p_Z$ is proper by
  base change, $\tilde q_Z$ is not of finite Tor amplitude and
  $(\tilde q_Z)^*$ must be defined as a {\it virtual} pullback
  \cite{Qu2018}. Along with the obvious projection from the bottom
  left to $\fM^{\doub}(\alpha) \times \fM^{\doub}(\beta)$, this makes
  $K_\sT(T^*\fM) \coloneqq \bigoplus_\alpha K_\sT(T^*\fM(\alpha))$
  into the {\it preprojective KHA} of $Q$.

  This is a K-theoretic version of the preprojective CoHA
  \cite{Yang2018}, and is conjecturally isomorphic \cite[Conjecture
  1.2]{Puadurariu2023} to the positive part of certain quantum loop
  algebras $U_q^+(L\lie{g}_Q)$. In \cite[Theorem 2.3.2]{Varagnolo2022}
  this is checked for $Q$ of finite or affine type excluding
  $A_1^{(1)}$.
\end{example}

\subsubsection{}

\begin{example}[Critical KHA, {\cite[\S 3]{Puadurariu2023}}] \label{ex:critical-KHA}
  Let $(Q, W)$ be a quiver with potential such that $\tr
  W_\alpha\colon \fM(\alpha) \to \bC$ is a regular function. The usual
  projections \eqref{eq:ext-stack-projections} from the Ext stack
  $\fM(\alpha,\beta)$ induce maps
  \begin{align*}
    K^{\crit}_\sT(\fM(\alpha) \times \fM(\beta), \tr (W_\alpha \boxplus W_\beta))
    &\xrightarrow{q^*} K^{\crit}_\sT(\fM(\alpha,\beta), \tr (p^*W_{\alpha+\beta})) \\
    &\xrightarrow{p_*} K^{\crit}_\sT(\fM(\alpha+\beta), \tr W_{\alpha+\beta})
  \end{align*}
  of critical K-groups, well-defined because one can easily check
  \[ \tr (p^*W_{\alpha+\beta}) = \tr q^*(W_\alpha \boxplus W_\beta). \]
  Pre-composed with $\boxtimes$, they make
  $K^{\crit}_\sT(\fM, \tr W) \coloneqq \bigoplus_\alpha K^{\crit}_\sT(\fM(\alpha), \tr W_\alpha)$
  into the {\it critical KHA} of $(Q, W)$.
\end{example}

\subsubsection{}

Consider the critical KHA for the tripled quiver $Q^{\trip}$ with
potential $\phi_\alpha$
(Lemma~\ref{lem:Mtrip-dimensional-reduction}), for which
\begin{equation} \label{eq:preprojective-to-critical-K-groups}
  K_\sT(T^*\fM(\alpha)) \cong K^{\crit}_\sT(\fM^{\trip}(\alpha), \phi_\alpha)
\end{equation}
as $\bk_{\sT \times \GL(\alpha)}$-modules by K-theoretic dimensional
reduction. However, the natural KHA structures on the two sides are
{\it not} isomorphic and a certain twist is required.

\begin{proposition}[{\cite[\S 3.2.2]{Puadurariu2023}}] \label{prop:twisted-preprojective-KHA}
  Let $\cE(\alpha)$ (resp. $\cE(\alpha,\beta)$) be the obvious
  projection $\fM^{\trip}(\alpha) \to \fM^{\doub}(\alpha)$ (resp.
  $\fM^{\trip}(\alpha,\beta) \to \fM^{\doub}(\alpha,\beta)$) viewed as
  a vector bundle. Set
  \[ \omega_{\alpha,\beta} \coloneqq \det\left(\cE(\alpha,\beta)/\cE(\alpha) \times \cE(\beta)\right). \]
  Use it to define the twist
  $\star_\omega \coloneqq p_*(\omega \otimes q^*(-))$ of the original
  preprojective Hall product $\star = p_*q^*$. Then
  \eqref{eq:preprojective-to-critical-K-groups} induces an isomorphism
  of KHAs
  \[ \left(K_\sT(T^*\fM), \star_\omega\right) \cong \left(K^{\crit}_\sT(\fM^{\trip}, \phi), \star\right). \]
\end{proposition}

\subsection{Product-coproduct compatibility}
\label{sec:compatibility}

\subsubsection{}

We first prove the compatibility theorem for the vertex coalgebra and
KHA structures on $K_\sT(\fM_Q)$
(Theorem~\ref{thm:compatibility-general}). Then we explain how to
modify the proof for the more complicated case of $K_\sT(T^*\fM_Q)$
(Theorem~\ref{thm:compatibility-preprojective}).

\subsubsection{}

\begin{theorem} \label{thm:compatibility-general}
  Let $V \coloneqq K_\sT(\fM_Q)$. On $V$, the vertex coalgebra
  structure $(\vac, D, \Y, C)$
  (Theorem~\ref{thm:coVA-construction-general}) and Hall product
  $\star$ (Example~\ref{ex:quiver-KHA}) form a commutative square
  \[ \begin{tikzcd}[column sep=huge]
      V \otimes V \ar{rr}{\check S^{(23)}(z) \circ (\Y(z) \boxtimes \Y(z))} \ar{d}[swap]{\star} & & (V \otimes V \otimes V \otimes V)\lseries*{(1-z)^{-1}} \ar{d}{\star \boxtimes \star} \\
      V \ar{rr}{\Y(z)} && (V \otimes V)\lseries*{(1-z)^{-1}}
    \end{tikzcd} \]
  where a superscript $(-)^{(ij)}$ means to act on the $i$-th and
  $j$-th factors, and
  \[ \check S_{\alpha,\beta}(z) \coloneqq \sigma_{12} \circ S_{\alpha,\beta}(z), \qquad S_{\alpha,\beta}(z) \coloneqq (-z)^{\rank \cE_{\alpha,\beta}} \det(\cE_{\alpha,\beta}) \Theta_{\beta,\alpha}(z) (\Theta_{\alpha,\beta}(z)^\vee)^{-1} \]
  is the braiding operator associated to the vertex coalgebra
  (Remark~\ref{rem:coVA-construction-general-braiding}).
\end{theorem}

\subsubsection{}

\begin{remark}
  This is a K-theoretic analogue of \cite{Latyntsev2021}, where a
  similar compatibility is shown for ordinary (nonequivariant)
  cohomology with its additive vertex coproduct and CoHA product. It
  is very easy to check that the unit for the Hall algebra, i.e. the
  generator of $K_\sT(\fM(0))$, is also compatible with the vertex
  coalgebra structure. Hence we call $V$ a {\it braided vertex
    bialgebra} following \cite{Li2007}, though we have an algebra
  structure on a vertex coalgebra rather than a coalgebra structure on
  a vertex algebra. These are distinct notions because not every
  (vertex) algebra induces a (vertex) coalgebra on the dual.
\end{remark}

\subsubsection{}

\begin{proof}
  Recall that the module $V$ is graded. Setting $V(\alpha) \coloneqq
  K_\sT(\fM(\alpha))$, it suffices to prove the commutativity of the
  graded piece
  \begin{equation} \label{eq:product-coproduct-compatibility-graded}
    \begin{tikzcd}[column sep=huge]
      V(\alpha) \otimes V(\beta) \ar{rr}{{\bigoplus \check S_{\alpha_2,\beta_1}^{(23)}(z) \circ \!\!\begin{array}{l} \Y_{\alpha_1,\alpha_2}(z) \boxtimes \\[-6pt] \Y_{\beta_1,\beta_2}(z)\end{array}}} \ar{d}[swap]{\star} && \bigoplus\left(\!\! \begin{array}{l} V(\alpha_1) \otimes V(\beta_1) \otimes \\ V(\alpha_2) \otimes V(\beta_2)\end{array}\!\!\right)\!\lseries*{(1-z)^{-1}} \ar{d}{\star \boxtimes \star} \\
      V(\alpha+\beta) \ar{rr}{\Y_{\gamma_1,\gamma_2}(z)} && (V(\gamma_1) \otimes V(\gamma_2))\lseries*{(1-z)^{-1}}
    \end{tikzcd}
  \end{equation}
  for given $(\alpha,\beta,\gamma_1,\gamma_2)$, where the sum
  $\bigoplus$ is over dimension vectors
  $(\alpha_1,\alpha_2,\beta_1,\beta_2)$ satisfying
  \begin{equation} \label{eq:product-coproduct-class-splitting}
    \begin{array}{cc} \alpha = \alpha_1+\alpha_2, & \gamma_1 = \alpha_1+\beta_1, \\ \beta = \beta_1+\beta_2, & \gamma_2 = \alpha_2+\beta_2. \end{array}
  \end{equation}

\subsubsection{}

  We follow the proof strategy of \cite[\S 10]{Latyntsev2021}. Consider
  the diagram
  \begin{equation} \label{eq:product-direct-sum-diagram}
    \begin{tikzcd}
      \bigsqcup \begin{array}{l} \fM(\alpha_1) \times \fM(\alpha_2) \times \\ \fM(\beta_1) \times \fM(\beta_2)\end{array} \ar{rr}{\Phi \times \Phi} && \fM(\alpha) \times \fM(\beta) \\
      \bigsqcup\; \fM(\alpha_1,\beta_1) \times \fM(\alpha_2,\beta_2) \ar{u}{\sigma_{23} \circ (q \times q)} \ar{dr}[swap]{p \times p} \ar[hookrightarrow]{r}{\iota} & \fM(\alpha,\beta)^{\text{split}}_{\gamma_1,\gamma_2} \ar{r}{\tilde\Phi} \ar{d}{\tilde p} & \fM(\alpha,\beta) \ar{d}{p} \ar{u}[swap]{q} \\
      & \fM(\gamma_1) \times \fM(\gamma_2) \ar{r}{\Phi} & \fM(\alpha+\beta).
    \end{tikzcd}
  \end{equation}
  where the disjoint unions $\bigsqcup$ range over all dimension
  vectors $(\alpha_1, \alpha_2, \beta_1, \beta_2)$ satisfying
  \eqref{eq:product-coproduct-class-splitting}, $\sigma_{23}$ swaps
  the second and third factors, and
  $\fM(\alpha,\beta)_{\gamma_1,\gamma_2}^{\text{split}}$ (and $\tilde
  p$ and $\tilde\Phi$) is defined by the bottom right square being a
  Cartesian square of {\it dg-stacks}. Explicitly,
  $\fM(\alpha,\beta)_{\gamma_1,\gamma_2}^{\text{split}}$ is a dg-stack
  which parameterizes tuples
  \begin{equation} \label{eq:extension-split-stack-objects}
    \left([0 \to A \to B \to C \to 0], B_1, B_2, g\right)
  \end{equation}
  where $[0 \to A \to B \to C \to 0] \in \fM(\alpha,\beta)$ is an
  extension, $B_i \in \fM(\gamma_i)$, and $g\colon B
  \xrightarrow{\sim} B_1 \oplus B_2$ is an isomorphism of objects in
  $\fM(\alpha+\beta)$. The embedding $\iota$ is of the locus where the
  extension is actually the direct sum of two extensions $0 \to A_i
  \to B_i \to C_i \to 0$ with $A_i \in \fM(\alpha_i)$ and $C_i \in
  \fM(\beta_i)$, and $g\colon B_1 \oplus B_2 \xrightarrow{\sim} B_1
  \oplus B_2$ is the identity (modulo automorphisms of the $B_i$).

\subsubsection{}

  The lower left triangle in \eqref{eq:product-direct-sum-diagram}
  consists of global quotients of $\sT$-equivariant dg-schemes by $G
  \coloneqq \GL(\gamma_1) \times \GL(\gamma_2)$, and $\sT$-equivariant
  morphisms between them. Since the $G$- and $\sT$-actions commute,
  and all potentials are $G$-invariant, for our purposes it may
  equivalently be considered as a triangle
  \begin{equation} \label{eq:extension-split-stack-fixed-locus}
    \begin{tikzcd}
      \bigsqcup \begin{array}{l} \left(M(\alpha_1,\beta_1) \times_{P(\alpha_1,\beta_1)} \GL(\gamma_1)\right) \times\\ \left(M(\alpha_2,\beta_2) \times_{P(\alpha_2,\beta_2)} \GL(\gamma_2)\right)\end{array} \ar{dr}[swap]{p \times p} \ar[hookrightarrow]{r}{\iota} & M(\alpha,\beta)_{\gamma_1,\gamma_2}^{\text{split}} \ar{d}{\tilde p} \\
                {} & M(\gamma_1) \times M(\gamma_2)
    \end{tikzcd}
  \end{equation}
  of $(\sT \times G)$-equivariant dg-schemes and $(\sT \times
  G)$-equivariant morphisms between them. Let $\bC^\times$ act on
  $M(\alpha,\beta)_{\gamma_1,\gamma_2}^{\text{split}}$ by scaling the
  $\gamma_1$ component, meaning that $\zeta \in \bC^\times$ acts on
  the tuple \eqref{eq:extension-split-stack-objects} by
  \[ \zeta \cdot \left([0 \to A \to B \to C \to 0], B_1, B_2, g\right) \coloneqq \left([0 \to A \to B \to C \to 0], B_1, B_2, (\zeta \oplus 1)g\right). \]
  This $\bC^\times$-action clearly commutes with the $(\sT \times
  G)$-action, and it is straightforward to check that $\iota$ is the
  inclusion of the $\bC^\times$-fixed locus.

\subsubsection{}

  We will verify the desired commutativity of
  \eqref{eq:product-coproduct-compatibility-graded} by direct
  computation using the diagram \eqref{eq:product-direct-sum-diagram}.
  Explicitly, the desired equality is
  \begin{equation} \label{eq:product-coproduct-compatibility-graded-explicit}
    \begin{aligned}
      &\sum (p_{\alpha_1,\beta_1} \times p_{\alpha_2,\beta_2})_* (q_{\alpha_1,\beta_1} \times q_{\alpha_2,\beta_2})^*\sigma_{23}^* \\
      &\qquad\left[S_{\alpha_2,\beta_1}(z) \otimes \wedge_{-z}^\bullet\left(\cE_{\alpha_1,\alpha_2}^\vee \boxplus \cE_{\beta_1,\beta_2}^\vee\right) \otimes (z^{\deg_1}\Phi_{\alpha_1,\alpha_2}^* \times z^{\deg_1}\Phi_{\beta_1,\beta_2}^*)E\right] \\
      &\stackrel{?}{=} \wedge_{-z}^\bullet\left(\cE_{\gamma_1,\gamma_2}^\vee\right) \otimes z^{\deg_1} \Phi_{\gamma_1,\gamma_2}^* (p_{\alpha,\beta})_* q_{\alpha,\beta}^*E,
    \end{aligned}
  \end{equation}
  where the sum ranges over all dimension vectors satisfying
  \eqref{eq:product-coproduct-class-splitting}, and $\cE$ is the
  bilinear element used to define $\Theta(z) =
  \wedge_{-z}^\bullet(\cE^\bullet)$. Note that $\cE$ is pulled back
  from a point and therefore tensor product with it commutes with all
  pushforwards and pullbacks.

\subsubsection{}

  We begin with the left hand side of
  \eqref{eq:product-coproduct-compatibility-graded-explicit}. We claim
  that
  \begin{align}
    &\sum (p_{\alpha_1,\beta_1} \times p_{\alpha_2,\beta_2})_* (q_{\alpha_1,\beta_1} \times q_{\alpha_2,\beta_2})^* \sigma_{23}^* \nonumber \\
    &\qquad\left[S_{\alpha_2,\beta_1}(z) \otimes \wedge_{-z}^\bullet\left(\cE_{\alpha_1,\alpha_2}^\vee \boxplus \cE_{\beta_1,\beta_2}^\vee\right) \otimes (z^{\deg_1}\Phi_{\alpha_1,\alpha_2}^* \times z^{\deg_1}\Phi_{\beta_1,\beta_2}^*)E\right] \nonumber \\
    &= \sum (p_{\alpha_1,\beta_1} \times p_{\alpha_2,\beta_2})_* S_{\alpha_2,\beta_1}(z) \otimes \wedge_{-z}^\bullet\left(\cE_{\alpha_1,\alpha_2}^\vee \boxplus \cE_{\beta_1,\beta_2}^\vee\right) \otimes z^{\deg_1} \iota_{\alpha_1,\beta_1,\alpha_2,\beta_2}^!\tilde\Phi_{\gamma_1,\gamma_2}^* q_{\alpha,\beta}^*E \nonumber \\
    &= \wedge_{-z}^\bullet\left(\cE_{\gamma_1,\gamma_2}^\vee\right) \otimes \sum (p_{\alpha_1,\beta_1} \times p_{\alpha_2,\beta_2})_* \frac{S_{\alpha_2,\beta_1}(z)}{\wedge_{-z}^\bullet\left(\cE_{\alpha_1,\beta_2}^\vee \boxplus \cE_{\beta_1,\alpha_2}^\vee\right)} z^{\deg_1} \iota_{\alpha_1,\beta_1,\alpha_2,\beta_2}^!\tilde\Phi^*q_{\alpha,\beta}^*E. \label{eq:product-coproduct-compatiblity-lhs}
  \end{align}
  where $\iota_{\alpha_1,\beta_1,\alpha_2,\beta_2}$ denotes the
  restriction of $\iota$ to the component $\fM(\alpha_1,\beta_1) \times
  \fM(\alpha_2,\beta_2)$. Namely, the first equality follows from the
  commutativity of the upper rectangle in
  \eqref{eq:product-direct-sum-diagram}, and the second equality follows
  from the bilinearity
  \[ (p_{\alpha_1,\beta_1} \times p_{\alpha_2,\beta_2})^*\cE_{\gamma_1,\gamma_2} = \cE_{\alpha_1,\alpha_2} \oplus \cE_{\alpha_1,\beta_2} \oplus \cE_{\beta_1,\alpha_2} \oplus \cE_{\beta_1,\beta_2} \]
  which is clear from the definition
  \eqref{eq:bilinear-bundle-moduli-stack} of $\cE$ (cf. the
  bilinearity \eqref{eq:bundle-bilinearity}). We omitted some
  pullbacks $(q \times q)^*\sigma_{23}^*$ on $S_{\alpha_2,\beta_1}(z)$
  and the various $\cE$ because $q^*$ is an isomorphism and the
  subscripts already make it clear which spaces each element is pulled
  back from. Note that the Gysin pullback $i^!$ is required because
  the usual pullback $i^*$ may not exist.

\subsubsection{}
\label{sec:compatibility-base-change-step}  

  Now we consider the right hand side of
  \eqref{eq:product-coproduct-compatibility-graded-explicit}. Since
  the lower right square in \eqref{eq:product-direct-sum-diagram} is
  Cartesian, by base change
  \[ \wedge_{-z}^\bullet\left(\cE_{\gamma_1,\gamma_2}^\vee\right) \otimes z^{\deg_1} \Phi_{\gamma_1,\gamma_2}^* (p_{\alpha,\beta})_* q_{\alpha,\beta}^*E = \wedge_{-z}^\bullet\left(\cE_{\gamma_1,\gamma_2}^\vee\right) \otimes z^{\deg_1} \tilde p_* \tilde\Phi^* q_{\alpha,\beta}^*E. \]
  Comparing with \eqref{eq:product-coproduct-compatiblity-lhs}, it
  therefore suffices to prove that
  \begin{equation} \label{eq:product-coproduct-compatiblity-localization}
    z^{\deg_1} \tilde p_*F \stackrel{?}{=} \sum (p_{\alpha_1,\beta_1} \times p_{\alpha_2,\beta_2})_* \frac{S_{\alpha_2,\beta_1}(z)}{\wedge_{-z}^\bullet\left(\cE_{\alpha_1,\beta_2}^\vee \boxplus \cE_{\beta_1,\alpha_2}^\vee\right)} z^{\deg_1} \iota_{\alpha_1,\beta_1,\alpha_2,\beta_2}^! F
  \end{equation}
  for any $F \in
  K_\sT(\fM(\alpha,\beta)_{\gamma_1,\gamma_2}^{\text{split}})$. This
  is an equality in $K_\sT(\fM(\gamma_1) \times
  \fM(\gamma_2))\lseries*{(1-z)^{-1}}$. We claim that it is a form of
  equivariant localization, as follows.

\subsubsection{}

  \begin{lemma} \label{lem:extension-split-stack-normal-bundle}
    The K-theory class of the relative tangent complex of $\iota$ is
    given by
    \[ \bT_{\iota_{\alpha_1,\beta_1,\alpha_2,\beta_2}} = -\cE_{\alpha_1,\beta_2} - \cE_{\alpha_2,\beta_1}. \]
  \end{lemma}

  \begin{proof}
    By the exact triangle for relative tangent complexes,
    \[ \bT_{\iota_{\alpha_1,\beta_1,\alpha_2,\beta_2}} = \bT_{p_{\alpha_1,\beta_1} \times p_{\alpha_2,\beta_2}} - \iota^* \bT_{\tilde p_{\alpha,\beta}} = \bT_{p_{\alpha_1,\beta_1} \times p_{\alpha_2,\beta_2}} - \iota^*\tilde\Phi^* \bT_{p_{\alpha,\beta}} \]
    where the second equality is base change for tangent complexes.
    Applying Lemma~\ref{lem:bilinear-bundle-as-relative-tangent}, this
    becomes
    \[ (\cE_{\alpha_1,\beta_1} + \cE_{\alpha_2,\beta_2}) - \cE_{\alpha_1+\alpha_2,\beta_1+\beta_2} = -\cE_{\alpha_1,\beta_2} - \cE_{\alpha_2,\beta_1} \]
    on $\fM(\alpha_1,\beta_1) \times \fM(\alpha_2,\beta_2)$, using the
    bilinearity \eqref{eq:bundle-bilinearity} of $\cE_{\alpha,\beta}$.
  \end{proof}

\subsubsection{}

  Let $\iota^{\bC^\times}$, $(p \times p)^{\bC^\times}$, and $\tilde
  p^{\bC^\times}$ denote the $(\bC^\times \times \sT \times
  G)$-equivariant versions of the maps in
  \eqref{eq:extension-split-stack-fixed-locus} and choose any
  $(\bC^\times \times \sT \times G)$-equivariant lift $F^{\bC^\times}$
  of $F$. Let $z$ denote the weight of the $\bC^\times$-action so
  that, for instance, $\bk_{\bC^\times \times \sT \times G} = \bk_{\sT
    \times G}[z^\pm]$. We work over the ring
  \begin{equation} \label{eq:extension-split-stack-localized-submodule}
    \bk_{\bC^\times \times \sT \times G}\left[\left(\wedge_{-1}^\bullet(z\cG)\right)^{-1} : \cG \in \cat{Coh}_{\sT \times G}(\pt)\right],
  \end{equation}
  in which $\wedge_{-1}^\bullet(z^\pm \cG)$ exists and is invertible.
  Using Lemma~\ref{lem:extension-split-stack-normal-bundle}, the
  virtual localization formula \eqref{eq:virtual-localization} with
  respect to the central subgroup $\bC^\times \subset \bC^\times
  \times \sT \times G$ says
  \begin{equation} \label{eq:extension-split-stack-localization}
    (\iota_*^{\bC^\times})^{-1}F^{\bC^\times} = \sum \left(\wedge_{-1}^\bullet(z\cE_{\alpha_1,\beta_2}^\vee \boxplus z^{-1}\cE_{\alpha_2,\beta_1}^\vee)\right)^{-1} (\iota^{\bC^\times}_{\alpha_1,\beta_1,\alpha_2,\beta_2})^! F^{\bC^\times}.
  \end{equation}
  Applying $((p \times p)^{\bC^\times})_*$ to both sides produces
  \begin{equation} \label{eq:extension-split-stack-localization-pushed}
    (\tilde p^{\bC^\times})_*F^{\bC^\times} = \sum ((p \times p)^{\bC^\times})_* \left(\wedge_{-1}^\bullet(z\cE_{\alpha_1,\beta_2}^\vee \boxplus z^{-1}\cE_{\alpha_2,\beta_1}^\vee)\right)^{-1} (\iota^{\bC^\times}_{\alpha_1,\beta_1,\alpha_2,\beta_2})^! F^{\bC^\times}
  \end{equation}
  by the commutativity of \eqref{eq:extension-split-stack-fixed-locus}.
  This is an equality in
  \[ K_{\sT \times G}(M(\gamma_1) \times M(\gamma_2))[z^\pm]\left[\left(\wedge_{-1}^\bullet(z\cG)\right)^{-1} : \cG \in \cat{Coh}_{\sT \times G}(\pt)\right]. \]

\subsubsection{}

  It remains to replace all $\bC^\times$-equivariant maps with their
  non-$\bC^\times$-equivariant versions, while still keeping track of
  $\bC^\times$-weights by applying $z^{\deg_1}$ and treating $z$ as a
  formal variable. This is valid because $\bC^\times$ acts trivially
  on $M(\gamma_1) \times M(\gamma_2)$. Hence
  \eqref{eq:extension-split-stack-localization-pushed} becomes
  \[ z^{\deg_1} \tilde p_*F = \sum (p \times p)_* \left(\wedge_{-1}^\bullet(z\cE_{\alpha_1,\beta_2}^\vee \boxplus z^{-1}\cE_{\alpha_2,\beta_1}^\vee)\right)^{-1} z^{\deg_1} \iota_{\alpha_1,\beta_1,\alpha_2,\beta_2}^!F. \]
  The localization factor may be rewritten as
  \begin{align*}
    \wedge_{-1}^\bullet(z\cE_{\alpha_1,\beta_2}^\vee \boxplus z^{-1}\cE_{\alpha_2,\beta_1}^\vee)
    &= (-z)^{-\rank \cE_{\alpha_2,\beta_1}} \det(\cE_{\alpha_2,\beta_1})^\vee \otimes \wedge_{-1}^\bullet(z\cE_{\alpha_1,\beta_2}^\vee \boxplus z\cE_{\alpha_2,\beta_1}) \\
    &= (-z)^{-\rank \cE_{\alpha_2,\beta_1}} \det(\cE_{\alpha_2,\beta_1})^\vee \wedge_{-1}^\bullet(z\cE_{\alpha_2,\beta_1})\wedge_{-1}^\bullet(z\cE_{\beta_1,\alpha_2}^\vee)^{-1} \\
    &\qquad\otimes \wedge_{-1}^\bullet(z\cE_{\alpha_1,\beta_2}^\vee \boxplus z\cE_{\beta_1,\alpha_2}^\vee).
  \end{align*}
  Finally, the expansion of Definition~\ref{sec:theta-expansion} may be
  applied to the inverses $(\wedge_{-1}^\bullet(z\cG))^{-1}$ in the
  ring \eqref{eq:extension-split-stack-localized-submodule}. The
  terms preceding $\otimes$ in the localization factor become exactly
  $S_{\alpha_2,\beta_1}(z)^{-1}$, by definition. The result is the
  desired identity
  \eqref{eq:product-coproduct-compatiblity-localization}.
\end{proof}

\subsubsection{}

\begin{theorem} \label{thm:compatibility-preprojective}
  Let $V \coloneqq K_\sT(T^*\fM_Q)_{\loc}$. On $V$, the vertex
  coalgebra structure $(\vac, D, \Y, C)$
  (Theorem~\ref{thm:coVA-construction-preprojective}) and Hall product
  $\star$ (Proposition~\ref{prop:twisted-preprojective-KHA}) form a
  commutative square
  \[ \begin{tikzcd}[column sep=huge]
      V \otimes V \ar{rr}{\check S^{(23)}(z) \circ (\Y(z) \boxtimes \Y(z))} \ar{d}[swap]{\star_\omega} & & (V \otimes V \otimes V \otimes V)\lseries*{(1-z)^{-1}} \ar{d}{\star_\omega \boxtimes \star_\omega} \\
      V \ar{rr}{\Y(z)} && (V \otimes V)\lseries*{(1-z)^{-1}}
    \end{tikzcd} \]
  where a superscript $(-)^{(ij)}$ means to act on the $i$-th and
  $j$-th factors, and
  \[ \check S_{\alpha,\beta}(z) \coloneqq \sigma_{12} \circ S_{\alpha,\beta}(z), \qquad S_{\alpha,\beta}(z) \coloneqq (-z)^{\rank \cE^{\trip}_{\alpha,\beta}} \det(\cE_{\alpha,\beta}^{\trip}) \Theta_{\beta,\alpha}^{\trip}(z) (\Theta_{\alpha,\beta}^{\trip}(z)^\vee)^{-1} \]
  is the braiding operator associated to the vertex coalgebra of
  $Q^{\trip}$ (Remark~\ref{rem:coVA-construction-general-braiding}).
\end{theorem}

\subsubsection{}

\begin{remark} \label{rem:vertex-bialgebra-general}
  In fact, one can verify that nothing in what follows depends on
  specific properties of $(\fM^{\trip}, \phi)$, which can be replaced
  by any $(\fM_Q, \tr W)$ as long as the critical K-group
  $K_\sT(\fM_Q, \tr W)$ satisfies a K\"unneth property (see
  Remark~\ref{rem:coVA-general}) so that the vertex coalgebra is
  well-defined. Under this assumption, the general result is that the
  critical KHAs of Example~\ref{ex:critical-KHA} become vertex
  bialgebras as well.
\end{remark}

\subsubsection{}

\begin{proof}[Proof of Theorem~\ref{thm:compatibility-preprojective}.]
  The proof of Theorem~\ref{thm:compatibility-general} may be adapted
  as follows.

  First, recall from
  \S\ref{sec:preprojective-vertex-coproduct-from-critical-K-theory}
  that the vertex coproduct $\Y$ on $V$ was actually defined using the
  $\bk_{\sT,\loc}$-module $K^{\crit}_\sT(\fM^{\trip}, \phi)_{\loc}$,
  which is isomorphic to $V$ by dimensional reduction. So, using
  Proposition~\ref{prop:twisted-preprojective-KHA}, we may consider $V
  = K^{\crit}_\sT(\fM^{\trip}, \phi)_{\loc}$ and the product $\star$,
  instead of $V = K_\sT(T^*\fM)_{\loc}$ and the product
  $\star_\omega$.

  Second, consider the diagram \eqref{eq:product-direct-sum-diagram}
  for $\fM^{\trip}$ instead of $\fM$. Using the potential
  $\phi_\alpha$ on $\fM^{\trip}(\alpha)$, we take the obvious choices
  of potentials on every term in the middle row of
  \eqref{eq:product-direct-sum-diagram} compatible with all the maps
  (see Example~\ref{ex:critical-KHA}). Since all stacks except the
  middle term
  $\fM^{\trip}(\alpha,\beta)_{\gamma_1,\gamma_2}^{\text{split}}$ are
  smooth, their critical K-groups with respect to these potentials are
  well-defined, and we want to prove
  \eqref{eq:product-coproduct-compatibility-graded-explicit}, as
  before.

  Finally, for a space $X$ with potential $\phi$, write $X_0 \coloneqq
  \phi^{-1}(0)$ for short. By the definition of critical K-theory, to
  prove an equality in $K_\sG^{\crit}(M, \phi) =
  K_\sG(M_0)/K_\sG^\circ(M_0)$, it suffices to prove it in the
  pre-quotient $K_\sG(M_0)$. We must therefore consider the diagram
  \[ \begin{tikzcd}[column sep=small]
      \bigsqcup \begin{array}{l} (\fM^{\trip}(\alpha_1) \times \fM^{\trip}(\alpha_2) \times \\ \fM^{\trip}(\beta_1) \times \fM^{\trip}(\beta_2))_0\end{array} \ar{rr}{\Phi \times \Phi} && (\fM^{\trip}(\alpha) \times \fM^{\trip}(\beta))_0 \\
      \bigsqcup\; (\fM^{\trip}(\alpha_1,\beta_1) \times \fM^{\trip}(\alpha_2,\beta_2))_0 \ar{u}{\sigma_{23} \circ (q \times q)} \ar{dr}[swap]{p \times p} \ar[hookrightarrow]{r}{\iota} & (\fM^{\trip}(\alpha,\beta)^{\text{split}}_{\gamma_1,\gamma_2})_0 \ar{r}{\tilde\Phi} \ar{d}{\tilde p} & \fM^{\trip}(\alpha,\beta)_0 \ar{d}{p} \ar{u}[swap]{q} \\
      & (\fM^{\trip}(\gamma_1) \times \fM^{\trip}(\gamma_2))_0 \ar{r}{\Phi} & \fM^{\trip}(\alpha+\beta)_0.
  \end{tikzcd} \]
  which is \eqref{eq:product-direct-sum-diagram} for $\fM^{\trip}$
  with all stacks replaced by the zero loci of their associated
  potentials. Using that
  \[ \begin{tikzcd}
    X_0 \ar[hookrightarrow]{d} \ar{r}{f_0} & Y_0 \ar[hookrightarrow]{d} \\
    X \ar{r}{f} & Y \ar{r}{\phi} & \bC
  \end{tikzcd} \]
  is a Cartesian square, and using various base change properties, it
  is straightforward to check that all steps in the proof of
  \eqref{eq:product-coproduct-compatibility-graded-explicit} continue
  to hold.
\end{proof}

\subsection{Comparison with ambient vertex bialgebra}
\label{sec:comparison-ambient-vertex-bialgebra}

\subsubsection{}

For a space $X$ with potential $\phi\colon X \to \bC$, write
$X_0 \coloneqq \phi^{-1}(0)$ for short.

\begin{theorem} \label{thm:PHA-to-ambient-comparison}
  Let $\fX = \fM^{\trip}$. The inclusion
  $i_0\colon \fX_0 \hookrightarrow \fX$ induces a vertex bialgebra
  morphism
  \begin{equation} \label{eq:PHA-to-ambient-map}
    i_{0*}\colon K_\sT^{\crit}(\fX, \phi)_{\loc} \to K_\sT(\fX)_{\loc}.
  \end{equation}
\end{theorem}

The content of this theorem is essentially the following three claims
about $i_{0*}$, which is what we will prove: it is well-defined, it
preserves the Hall products, and it preserves vertex coproducts.
Recall that it was necessary to work over $\bk_{\sT,\loc}$ to define
the vertex coproduct on the left hand side (see
\S\ref{sec:coVA-preprojective-issues}); in contrast, none of these
claims actually requires this localization in a crucial way.

\subsubsection{}

\begin{lemma}[{\cite[Proposition 3.6]{Puadurariu2023}}] \label{lem:PHA-to-ambient-map-well-defined}
  The morphism \eqref{eq:PHA-to-ambient-map} is well-defined, even
  without localization.
\end{lemma}

\begin{proof}
  We must show that the image of $K^\circ_\sT(\fX_0)$ is killed by
  $i_{0*}\colon K_\sT(\fX_0) \to K_\sT(\fX)$. Since each $\phi_\alpha$
  is non-zero, $i_0$ is a regular embedding and
  \[ i_0^*i_{0*} = (1 - w) \cdot \id = 0 \]
  on $K^\circ_\sT(\fX_0)$, where $w = 1$ is the $\sT$-weight of
  the potentials $\phi_\alpha$. Hence it suffices to show $i_0^*\colon
  K_\sT(\fX) = K_\sT^\circ(\fX) \to K_\sT^\circ(\fX_0)$ is injective.

  Write $\fX(\alpha) = [X(\alpha)/\GL(\alpha)]$. The fixed locus
  $\iota\colon X(\alpha)^{\bC^\times_\hbar} \hookrightarrow X(\alpha)$
  is smooth because $X(\alpha)$ is smooth. By equivariant
  localization,
  \[ \iota^*\colon K_{\sT \times \GL(\alpha)}^\circ(X(\alpha))_{\loc} \to K_{\sT \times \GL(\alpha)}^\circ(X(\alpha)^{\bC^\times_\hbar})_{\loc} \]
  is an isomorphism. But clearly $\iota$ factors as
  \[ \iota\colon X(\alpha)^{\bC^\times_\hbar} \hookrightarrow X(\alpha)_0 \xhookrightarrow{i_0} X(\alpha), \]
  and all pullbacks exist in $K^\circ$ and are functorial, so
  $i_0^*\colon K_\sT^\circ(\fX(\alpha))_{\loc} \to
  K_\sT^\circ(\fX(\alpha)_0)_{\loc}$ must be injective. Finally, since
  $K_{\sT \times \GL(\alpha)}^\circ(X(\alpha)) \hookrightarrow K_{\sT
    \times \GL(\alpha)}^\circ(X(\alpha))_{\loc}$ is injective by
  direct computation, the original $i_0^*\colon K_\sT^\circ(\fX) \to
  K_\sT^\circ(\fX_0)$ must also be injective.
\end{proof}

\subsubsection{}

\begin{lemma}[{\cite[Proposition 3.6]{Puadurariu2023}}] \label{lem:ambient-KHA-compatibility}
  The morphism \eqref{eq:PHA-to-ambient-map} is an algebra morphism.
\end{lemma}

\begin{proof}
  Clearly $i_{0*}$ preserves the unit. For the Hall product, since
  $K_\sT^{\crit}(M, \phi)$ is a quotient of $K_\sT(M_0)$ by
  definition, the K\"unneth property and
  Lemma~\ref{lem:PHA-to-ambient-map-well-defined} imply that it
  suffices to show the following diagram commutes:
  \[ \begin{tikzcd}
      K_\sT\left((\fX(\alpha) \times \fX(\beta))_0\right) \ar{r}{q^*} \ar{d}{i_{0*}} & K_\sT(\fX(\alpha,\beta)_0) \ar{r}{p_*} \ar{d}{i_{0*}} & K_\sT(\fX(\alpha+\beta)_0) \ar{d}{i_{0*}} \\
      K_\sT(\fX(\alpha) \times \fX(\beta)) \ar{r}{q^*} & K_\sT(\fX(\alpha,\beta))_{\loc} \ar{r}{p_*} & K_\sT(\fX(\alpha+\beta)).
    \end{tikzcd} \]
  The left square commutes by base change, and the right square
  commutes by functoriality. So $i_{0*}$ preserves the Hall product.
\end{proof}

\subsubsection{}

\begin{lemma}
  The morphism \eqref{eq:PHA-to-ambient-map} is a vertex coalgebra
  morphism.
\end{lemma}

\begin{proof}
  Clearly $i_{0*}$ preserves the covacuum. Also, since the bilinear
  element $\cE_{\alpha,\beta}$ is pulled back from $\bk_\sT$, as a
  $\bk_\sT$-module homomorphism $i_{0*}$ automatically commutes with
  tensor product by $\Theta(z)$. It remains to show that $i_{0*}$ is
  compatible with pullbacks along the direct sum map $\Phi$, as well
  as the scaling automorphism map $\Psi$ used to construct the
  translation operator $z^{\deg}$. Such compatibilities follow from
  the same base change argument as in the proof of
  Lemma~\ref{lem:ambient-KHA-compatibility}.
\end{proof}

\subsubsection{}
\label{sec:quiver-KHA-shuffle-formula}

For explicit computations, we record here some formulas for the vertex
bialgebra $K_\sT(\fX)$. First, let $a_e \in \bk_\sA$ be the weight of
the edge $e$ in $Q$. Then in $K_\sT(\fX(\alpha) \times \fX(\beta))$,
\begin{equation} \label{eq:bilinear-bundle-Mtrip}
  \cE_{\alpha,\beta} = \sum_{e\colon i \to j} \left[a_e \cV_{\alpha,i}^\vee \boxtimes \cV_{\beta,j} + \frac{\hbar}{a_e} \cV_{\alpha,j}^\vee \boxtimes \cV_{\beta,i}\right] + \left(\frac{1}{\hbar} - 1\right) \sum_i \cV_{\alpha,i}^\vee \boxtimes \cV_{\beta,i}
\end{equation}
where the sum is over edges of the quiver $Q$ (not $Q^{\doub}$ or
$Q^{\trip}$). In what follows we implicitly identify
\[ K_\sT(\fX(\alpha+\beta)) = \bk_\sT[s_{\alpha+\beta,i,j}]^{S(\alpha+\beta)} \subset \bk_\sT[s_{\alpha,i,j}]^{S(\alpha)}[s_{\beta,i,j}]^{S(\beta)} = K_\sT(\fX(\alpha) \times \fX(\beta)) \]
using $s_{\alpha,i,j} \leftrightarrow s_{\alpha+\beta,i,j}$ and
$s_{\beta,i,j} \leftrightarrow s_{\alpha+\beta,i,j+\alpha_i}$. This
makes sense of tautological bundles like
$\cV_{\alpha,i} = \sum_j s_{\alpha,i,j}$ whenever they appear on
$\fX(\alpha+\beta)$, such as in \eqref{eq:hall-product-quivers} below.

The vertex coproduct of the Laurent polynomial
$h \in K_\sT(\fX(\alpha+\beta))$, viewed as a function of variables
$s_{\alpha,i,j}$ and $s_{\beta,i,j}$, is
\begin{equation} \label{eq:vertex-coproduct-quivers}
  \Y_{\alpha,\beta}(z) h = h\Big|_{s_{\alpha,i,j} \mapsto z s_{\alpha,i,j}} \cdot \wedge_{-z}^\bullet \cE_{\alpha,\beta}^\vee
\end{equation}
for the appropriate expansion in $z$ (\S\ref{sec:theta-expansion}).
The Hall product of the Laurent polynomials $f \in K_\sT(\fX(\alpha))$
and $g \in K_\sT(\fX(\beta))$ is, by localization on $\GL/P$ or
otherwise,
\begin{equation} \label{eq:hall-product-quivers}
  \begin{aligned}
    f \star g
    &= \sum_{w \in S(\alpha+\beta)/S(\alpha)\times S(\beta)} w \cdot \left(fg \frac{\wedge^\bullet_{-1}(\cN_i^\vee)}{\wedge^\bullet_{-1}(\sum_i \cV_{\beta,i}^\vee \otimes \cV_{\alpha,i})}\right) \\
    &= \frac{1}{\alpha! \beta!} \sum_{w \in S(\alpha+\beta)} w \cdot \left(\frac{fg}{\wedge^\bullet_{-1}(\cE_{\alpha,\beta}^\vee)}\right)
  \end{aligned}
\end{equation}
where, with the factorization \eqref{eq:HA-p} of $p$ in mind, $\cN_i$
is the normal bundle of the map $i$ and the denominator is the
localization weight of $\pi$. In spite of the denominator, we know a
priori that the result lands in the Laurent polynomial ring $\bk_{\sT
  \times \GL(\alpha+\beta)}$. The second equality follows from
Lemma~\ref{lem:bilinear-bundle-as-relative-tangent} and that $f$ and
$g$ are already $S(\alpha)$- and $S(\beta)$-symmetric respectively.
Formulas like \eqref{eq:hall-product-quivers} are known as {\it
  shuffle products}, and the non-trivial rational function being
multipled to $f$ and $g$ is called the {\it kernel}. See \cite[\S
  2]{Kontsevich2011} for more explicit examples.

\subsubsection{}

\begin{remark}
  After base change to the fraction field of $\bk_\sT$, it is known
  \cite[Corollary 2.16]{Negut2021} that $i_{0*}$ is injective with
  image characterized by those Laurent polynomials $f(s_{\alpha,i,k})$
  satisfying the {\it wheel condition}
  \[ f\Big|_{a_e s_{\alpha,i,k_1} = \hbar s_{\alpha,j,k_2} = \hbar a_e s_{\alpha,i,k_3}} = f\Big|_{s_{\alpha,j,k_1} = a_e s_{\alpha,i,k_2} = \hbar s_{\alpha,j,k_3}} = 0 \]
  for all edges $e\colon i \to j$ in $Q$ and all $k_1 \neq k_3$ (and
  further $k_1 \neq k_2 \neq k_3$ if $i = j$). It is a straightforward
  exercise to verify algebraically that the Hall product $\star$
  preserves the wheel condition. As a much more trivial observation
  and sanity-check, the vertex coproduct
  \eqref{eq:vertex-coproduct-quivers} also preserves the wheel
  condition.
\end{remark}

\appendix

\section{K\"unneth property in K-theory}
\label{sec:kunneth-properties}

\subsubsection{}

In this appendix, we provide a general strategy
(Theorem~\ref{thm:kunneth-inductive}) to prove K\"unneth properties of
equivariant K-groups of spaces $X$, assuming that $X$ admits a
stratification where the equivariant K-groups of each stratum have
K\"unneth-like properties. In particular, in
Example~\ref{ex:nilpotent-endomorphism-stack-stratification}, we apply
this strategy to the moduli stack $\fN^{\nil}(\alpha)$
(Definition~\ref{def:endomorphism-stacks}) of nilpotent endomorphisms.

Throughout, whenever there is a scheme $X$ acted on by an algebraic
group $\sG$, we assume $X$ is quasi-projective and $G$ is reductive.

\subsubsection{}
\label{sec:BM-and-chow-setup}

Let $H^\sG(-)$ (resp. $A^\sG(-)$) denote $\sG$-equivariant
Borel--Moore homology (resp. Chow homology) with rational coefficients
and let $\bh^\sG \coloneqq H^\sG(\pt)$ be the base ring. Recall that
this means to take ordinary Borel--Moore or Chow homology of an
algebraic approximation to the topological realization
$X^{\cat{top}}_{\sG} \coloneqq X \times_\sG E\sG$ of the stack
$[X/\sG]$ \cite[\S 2.7]{Edidin1998}. In particular, both $H^\sG(-)$
and $A^\sG(-)$ retain the properties in
\S\ref{sec:equivariant-K-theory-properties}, e.g. Thom isomorphism
(with a degree shift).

Let $\hat H^\sG(-) \coloneqq \prod_{i \ge 0} H^\sG_i(-)$ denote completion
with respect to degree and similarly for $\hat H^\sG_{\alg}$. Similarly
define $\hat A^\sG(-)$. Finally, let $I_\sG \subset \bk_\sG$ be the
augmentation ideal and let $\hat K_\sG(-)$ denote the $I_\sG$-adic
completion of $K_\sG(-)$. We will use the composition
\begin{equation} \label{eq:equivariant-cycle-map}
  K_\sG(-) \to \hat K_\sG(-) \xrightarrow{\tau} \hat A_\sG(-) \xrightarrow{\cl} \hat H_\sG(-)
\end{equation}
where $\tau$ denotes the equivariant Riemann--Roch morphism
\cite[Theorem 4]{Edidin1998} and $\cl$ is the cycle class morphism.
Both $\tau$ and $\cl$ inherit the same properties as their
non-equivariant counterparts. For us, $\cl$ will always be an
isomorphism.

\subsubsection{}

\begin{example} \label{ex:k-theory-to-homology-of-point}
  Let $X = \pt$ and $\sG = \GL(n)$. This is essentially the only case
  of \eqref{eq:equivariant-cycle-map} of relevance to us.
  \begin{itemize}
  \item The $I_{\GL(n)}$-adic completion of
    $K_{\GL(n)}(\pt) = \bZ[s_1^\pm, \ldots, s_n^\pm]^{S_n}$ is
    \[ \hat K_{\GL(n)}(\pt) = \bZ\pseries*{1-s_1, \ldots, 1-s_n}^{S_n}. \]
  \item The topological realization
    $\pt^{\cat{top}}_{\GL(n)} = \varinjlim_N \Gr(n, N)$ is the
    infinite Grassmannian, with
    \[ \hat A^{\GL(n)}(\pt) = \hat H^{\GL(n)}(\pt) = \bQ\pseries*{u_1, \ldots, u_n}^{S_n}. \]
    The cycle class map $\cl$ is an isomorphism.
  \item The equivariant Riemann--Roch map $\tau$ is given by
    $s_i \mapsto \exp(u_i)$. This yields an isomorphism
    $\bQ\pseries*{1-s_i} \cong \bQ\pseries*{u_i}$, as one would
    expect.
  \end{itemize}
  Importantly, the composition \eqref{eq:equivariant-cycle-map} is
  therefore injective.

  For $\sG = \prod_k \GL(n_k)$, the same calculation holds but with
  multiple sets of (independently) symmetrized variables.
\end{example}

\subsubsection{}

\begin{remark}
  Equivariant Borel--Moore and Chow homology can be defined for
  arbitrary algebraic stacks --- in fact, even for derived stacks
  \cite[\S 2.2]{Aranha2024a} --- and so we take the liberty of stating
  the main Theorem~\ref{thm:kunneth-inductive} in this generality. But
  we will only apply it in the case where $\fX = [X/G]$ and $\fY =
  [Y/H]$ are global quotients where the $\sG$-action on $\fX$ and
  $\fY$ is induced from an $\sG$-action on $X$ and $Y$ which commutes
  with the $G$ and $H$ actions respectively. In this setting, the
  definitions and content of \S\ref{sec:BM-and-chow-setup} apply.
\end{remark}

\subsubsection{}

\begin{theorem} \label{thm:kunneth-inductive}
  Let $\sG$ be an algebraic group acting on algebraic stacks $\fX$ and
  $\fY$, and assume $\hat H^{\sG}(\fY)$ is flat over $\hat \bh^{\sG}$.
  Let
  \[ \fZ \xhookrightarrow{i} \fX \xhookleftarrow{j} \fU \]
  be inclusions of a $\sG$-invariant substack $\fZ$ and its
  complement $\fU$. Suppose, for both $\fZ$ and $\fU$:
  \begin{enumerate}
  \item $\boxtimes\colon K_\sG(-) \otimes_{\bk_\sG} K_{\sG}(\fY) \to
    K_{\sG}(- \times \fY)$ is surjective;
  \item $H^G(-)$ is a free $\bh^\sG$-module which is zero in odd
    degree;
  \item $K_\sG(-) \otimes_{\bk_\sG} K_{\sG}(\fY) \to \hat H^\sG(-)
    \otimes_{\hat \bh^\sG} \hat H^{\sG}(\fY)$ is injective.
  \end{enumerate}
  Then the same are true for $\fX$. Furthermore, properties (ii) and
  (iii) imply:
  \begin{enumerate}[resume]
  \item $\boxtimes\colon K_\sG(-) \otimes_{\bk_\sG} K_{\sG}(\fY) \to
    K_{\sG}(- \times \fY)$ is injective.
  \end{enumerate}
\end{theorem}

\begin{proof}
  (i) The four lemma implies the middle vertical arrow in
  \[ \begin{tikzcd}
      K_\sG(\fZ) \otimes_{\bk_\sG} K_{\sG}(\fY) \ar{r} \ar{d} & K_\sG(\fX) \otimes_{\bk_\sG} K_\sG(\fY) \ar{r} \ar{d} & K_\sG(\fU) \otimes_{\bk_\sG} K_\sG(\fY) \ar[twoheadrightarrow]{d} \ar{r} & 0 \\
      K_\sG(\fZ \times \fY) \ar{r} & K_\sG(\fX \times \fY) \ar{r} & K_\sG(\fU \times \fY) \ar{r} & 0
    \end{tikzcd} \]
  is surjective, where the rows arise from the long exact sequences in
  K-theory for $\fZ \hookrightarrow \fX \hookleftarrow \fU$ and $\fZ
  \times \fY \hookrightarrow \fX \times \fY \hookleftarrow \fU \times
  \fY$ and the vertical arrows are $\boxtimes$.
  
  (ii) The long exact sequence in Borel--Moore homology for $\fZ
  \hookrightarrow \fX \hookleftarrow \fU$ breaks into short exact
  sequences and yields the short exact sequence
  \[ 0 \to H^{\sG}(\fZ) \to H^{\sG}(\fX) \to H^{\sG}(\fU) \to 0 \]
  because $H^{\sG}_{\text{odd}}(\fU) = 0 = H^{\sG}_{\text{odd}}(\fZ)$
  by hypothesis. It splits since $H^{\sG}(\fU)$ is free over
  $\bh^{\sG}$.

  (iii) The (other) four lemma implies the middle arrow in
  \[ \begin{tikzcd}
      {} & K_\sG(\fZ) \otimes_{\bk_\sG} K_\sG(\fY) \ar{r} \ar[hookrightarrow]{d} & K_\sG(\fX) \otimes_{\bk_\sG} K_\sG(\fY) \ar{r} \ar{d} & K_\sG(\fU) \otimes_{\bk_\sG} K_\sG(\fY) \ar[hookrightarrow]{d} \\
      0 \ar{r} & \hat H^\sG(\fZ) \otimes_{\hat \bh^\sG} \hat H^\sG(\fY) \ar{r} & \hat H^\sG(\fX) \otimes_{\hat \bh^\sG} \hat H^\sG(\fY) \ar{r} & \hat H^\sG(\fU) \otimes_{\hat \bh^\sG} \hat H^\sG(\fY)
    \end{tikzcd} \]
  is injective, where the rows are induced from the long exact
  sequences in K-theory and Borel--Moore homology for $\fZ
  \hookrightarrow \fX \hookleftarrow \fU$. The bottom left arrow is
  injective since $\hat H^G_{\text{odd}}(U) = 0$ and tensor product
  with the flat $\bh^{\sG}$-module $\hat H^{\sG}(\fY)$ is exact.

  (iv) Using either property in (ii), the Eilenberg--Moore spectral
  sequence in Borel--Moore homology for $\fX \times \fY$ clearly
  degenerates, hence the bottom arrow in the commutative square
  \[ \begin{tikzcd}
      K_\sG(\fX) \otimes_{\bk_\sG} K_\sG(\fY) \ar{r} \ar[hookrightarrow]{d} & K_\sG(\fX \times \fY) \ar{d} \\
      \hat H^\sG(\fX) \otimes_{\hat \bh^\sG} \hat H^\sG(\fY) \ar[hookrightarrow]{r} & \hat H^\sG(\fX \times \fY)
    \end{tikzcd} \]
  is injective. By property (iii) so is the left vertical arrow. So
  the top arrow must also be injective.
\end{proof}

\subsubsection{}

\begin{corollary} \label{cor:kunneth-inductive-application}
  Suppose $\fX_1$ and $\fX_2$ are algebraic stacks with $\sG$-action,
  and both admit decompositions into finitely many disjoint locally
  closed $\sG$-invariant strata of the form $[\bC^N/G]$ such that:
  \begin{enumerate}
  \item $G$ is a unipotent extension of a product of general linear
    groups;
  \item the $\sG$-action on $[\bC^N/G]$ is induced from a $\sG$-action
    on $\bC^N$ commuting with the $G$-action.
  \end{enumerate}
  Then exterior tensor product induces an isomorphism
  \[ \boxtimes\colon K_\sG(\fX_1) \otimes_{\bk_\sG} K_\sG(\fX_2) \xrightarrow{\sim} K_\sG(\fX_1 \times \fX_2). \]
\end{corollary}

\begin{proof}
  Fix a stratum $\fU_i \cong [\bC^N/G]$ of $\fX_1$. By Thom
  isomorphism and its analogue for Borel--Moore homology,
  \begin{equation} \label{eq:stacky-affine-cells}
    \begin{aligned}
    K_\sG(\fU_i \times \fW) &\cong K_\sG([\pt/G] \times \fW) = K_{\sG \times G}(\fW) \\
    \hat H^\sG(\fU_i \times \fW) &\cong \hat H^\sG([\pt/G] \times \fW) = \hat H^{\sG \times G}(\fW)
    \end{aligned}
  \end{equation}
  for any algebraic stack $\fW$ with $\sG$-action. (On the right hand
  side, $G$ acts trivially on $\fW$.) We use this to check that
  $\fU_i$ satisfies properties (i), (ii) and (iii) of the theorem with
  $\fY = \fV_j$ where $\fV_j \cong [\bC^{N'}/G']$ is a stratum of
  $\fX_2$.
  \begin{enumerate}
  \item Compare $\fW = \pt$ with arbitrary $\fW$ in
    \eqref{eq:stacky-affine-cells} to see that $\boxtimes\colon
    K_\sG(\fU_i) \otimes_{\bk_\sG} K_\sG(\fW) \to K_\sG(\fU_i \times
    \fW)$ is an isomorphism. In particular this holds for $\fW =
    \fV_j$.
  \item Take $\fW = \pt$ in \eqref{eq:stacky-affine-cells} and apply
    the Borel--Moore analogue of \eqref{eq:equivariant-reduction} to
    reduce to the case where $G$ is actually a product of general
    linear groups. By Example~\ref{ex:k-theory-to-homology-of-point},
    $\hat H^\sG(\fU_i)$ is isomorphic as a $\hat \bh^\sG$-module to a
    free power series ring over $\hat \bh^\sG$ with all generators in
    even degree.
  \item Take $\fW = \pt$ in \eqref{eq:stacky-affine-cells}. By
    explicit computation following
    Example~\ref{ex:k-theory-to-homology-of-point}, the map
    $K_\sG(\fU_i) \otimes_{\bk_\sG} K_\sG(\fV_j) \hookrightarrow \hat
    H^\sG(\fU_i) \otimes_{\hat \bh^\sG} \hat H^\sG(\fV_j)$ is
    injective.
  \end{enumerate}

  Now use double induction on the stratifications $\fX_1 =
  \bigsqcup_{i=1}^n \fU_i$ and $\fX_2 = \bigsqcup_{j=1}^m \fV_j$.
  Namely, let $P(I,J)$ be the statement ``properties (i), (ii), and
  (iii) hold for $\fZ = \bigsqcup_{i \in I} \fU_i$ and $\fY =
  \bigsqcup_{j \in J} \fV_j$''. We just proved the base cases
  $P(\{i\},\{j\})$ for all $i$ and $j$. The theorem provides the
  inductive step for $I$, and then also for $J$ by exchanging the
  roles of $\fX$ and $\fY$. The hypothesis that $\hat H^\sG(\fY)$ is
  flat over $\hat \bh^\sG$ is always satisfied by property (ii) from
  an earlier inductive step, since it implies that $\hat H^\sG(\fY)$
  is in fact free over $\hat \bh^\sG$.

  We conclude by induction that $P(\{1,\ldots,n\},\{1,\ldots,m\})$
  holds. In particular, properties (i) and (iv) say that
  $\boxtimes\colon K_\sG(\fX) \otimes_{\bk_\sG} K_\sG(\fY) \to
  K_\sG(\fX \times \fY)$ is both injective and surjective.
\end{proof}

\subsubsection{}

\begin{example} \label{ex:nilpotent-endomorphism-stack-stratification}
  Consider the moduli stack $\fN^{\nil}(\alpha)$
  (Definition~\ref{def:endomorphism-stacks}) of nilpotent
  endomorphisms. Following standard ideas, see e.g. \cite[Theorem
    3.4]{Davison2018}, we may stratify $\fN^{\nil}(\alpha)$ by the
  Jordan type of $x^\circ$. View $x^\circ$ as a sequence of
  surjections
  \[ x \coloneqq x_0 \xrightarrow{x^\circ} x_1 \xrightarrow{x^\circ} x_2 \xrightarrow{x^\circ} \cdots \]
  with $x_{j+1} \coloneqq \im(x^\circ\big|_{x_j})$. Then the strata
  are the loci where the graded pieces have prescribed dimensions
  $\gamma_j = \dim x_j/x_{j+1}$ (which sum to $\alpha$). Each stratum
  is therefore an iterated $\Ext$ bundle over bases of the form
  $\prod_j \fM(\beta_j)$. So we may apply
  Corollary~\ref{cor:kunneth-inductive-application}, with $\sG
  \coloneqq \sT$, to conclude that
  \[ \boxtimes\colon K_\sT(\fN^{\nil}(\alpha)) \otimes_{\bk_\sT} K_\sT(\fN^{\nil}(\beta)) \to K_\sT(\fN^{\nil}(\alpha) \times \fN^{\nil}(\beta)) \]
  is an isomorphism.
\end{example}

\subsubsection{}

\begin{remark}
  The entire moduli stack $\fN(\alpha)$, not just $\fN^{\nil}(\alpha)
  \subset \fN(\alpha)$, may be stratified according to the Jordan type
  of the endomorphism $x^\circ$. To be precise, given a decomposition
  \[ \alpha = \sum_{i=1}^n m_i \alpha^{(i)} \]
  into pairwise distinct dimension vectors $\vec\alpha \coloneqq
  (\alpha^{(i)})_{i=1}^n$ and positive integer multiplicities $\vec m
  \coloneqq (m_i)_{i=1}^n$, consider the moduli substack
  \[ \fN_{\vec m, \vec\alpha} \subset \fN(\alpha) \]
  parameterizing $(x, x^\circ)$ such that
  \[ x \cong \bigoplus_{i=1}^n (x_{i,1} \oplus \cdots \oplus x_{i,m_i}) \]
  where $x_{i,j} \in \fM(\alpha^{(i)})$, and $x^\circ$ acts on
  $x_{i,j}$ with (generalized) eigenvalue $\lambda_{i,j}$, such that
  $\lambda_{i,j} \neq \lambda_{i,k}$ for any $1 \le j \neq k \le m_i$.
  Then $\fN_{\vec m, \vec\alpha}$, ranging over all choices of $n$,
  $\vec m$ and $\vec\alpha$, form a stratification of $\fN(\alpha)$;
  the condition on eigenvalues is to prevent these strata from
  overlapping. Explicitly,
  \[ \fN_{\vec m, \vec\alpha} \cong \prod_{i=1}^n \fN^{\nil}(\alpha^{(i)})^{\times m_i} \times U_{m_i} \]
  where $U_m \subset \bC^m$ is the complement of the union of all
  diagonals. However, in contrast to
  Example~\ref{ex:nilpotent-endomorphism-stack-stratification},
  Theorem~\ref{thm:kunneth-inductive} does not apply to this
  stratification because $U_m$ typically has odd Borel--Moore
  homology. For instance, the complement of the diagonal in $\bC^2$
  has non-trivial $H_3$.
\end{remark}

\phantomsection
\addcontentsline{toc}{section}{References}

\begin{small}
\bibliographystyle{alpha}
\bibliography{VAQG}
\end{small}

Kavli Institute for the Physics and Mathematics of the Universe (WPI), The University of Tokyo Institutes for Advanced Study, The University of Tokyo, Kashiwa, Chiba 277-8583, Japan

\textit{E-mail address}: \texttt{\href{mailto:henry.liu@ipmu.jp}{henry.liu@ipmu.jp}}

\end{document}